\RequirePackage{fix-cm}
\documentclass{svjour3}                     
\smartqed  
\usepackage{my}
\usepackage{graphicx}
\usepackage[letterpaper,top=1in, bottom=1in, left=1.5in, right=1.5in]{geometry}
\begin{document}

\title{ANAPT: Additive Noise Analysis for Persistence Thresholding}
\titlerunning{ANAPT}

\author{Audun D. Myers         \and
        Firas A. Khasawneh 	\and
        Brittany T. Fasy
}

\institute{A. Myers \at
              Michigan State University, College of Engineering \\
              \email{myersau3@msu.edu}           
           \and
           F. Khasawneh \at
              Michigan State University, College of Engineering \\
              \email{khasawn3@egr.msu.edu}           
           \and
           B. Fasy \at
              Montana State University, School of Computing \& Dept. of
              Mathematical Sciences \\
              \email{brittany.fasy@montana.edu}           
}

\date{Preprint as of \today}

\maketitle

\begin{abstract}

We introduce a novel method for Additive Noise Analysis for Persistence Thresholding (ANAPT) which separates significant features in the sublevel
set persistence diagram of a time series based on a statistics analysis of the 
persistence of a noise distribution. 
Specifically, we consider an additive noise model and leverage the statistical analysis to provide a noise cutoff or confidence interval in the persistence diagram for the observed time series. 
This analysis is done for several common noise models including Gaussian, uniform, exponential and Rayleigh distributions.
ANAPT is computationally efficient, does not require any signal pre-filtering, is widely applicable, and has open-source software available. 
We demonstrate the functionality ANAPT with both numerically simulated examples and an experimental data set. 
Additionally, we provide an efficient $\Theta(n\log(n))$ algorithm for calculating the zero-dimensional sublevel set persistence homology.
\end{abstract}

\section{Introduction} \label{sec:intro}

In this work, we are interested in separating noise from signal in time series
data. We approach this task through studying the statistics of sublevel set
persistence for several common noise models.  Specifically, we are using
persistence as a tool from Topological Data Analysis (TDA), which is a method of data
analysis~\cite{Edelsbrunner2002} used to study datasets from the
perspective of topology or shape. This intersection between TDA and time series
analysis is known as Topological Signal Processing (TSP). TSP includes
mathematical tools for analyzing signals from dynamical systems through TDA,
which includes applications such as dimensionality reduction, detecting
bifurcations in the system behavior, or dynamic state detection (regular or
chaotic).
TSP methods help study these applications by leveraging the mathematics of
dynamical systems theory, algebraic topology, information theory, and graph
theory.
Because of the novel tools used in TSP, it can be used to reveal information
that may not be currently accessible through standard dynamic systems methods.

In general, there are two common methods for using TSP to study single variable
time series. The first is to study the sublevel sets of a time series and how
they persist as the height of the sublevel set is increased. The second is to
transform the time series into a shape in Euclidean space  (e.g., a sliding
window embedding~\cite{Takens1981} creates an shape known as the \emph{phase
space}) and
study the shape through the Vietoris--Rips filtration.
Studying the shape of the phase space has been successfully
used for periodicity quantification with applications in signal
analysis~\cite{Perea2014}, detecting quasiperidoicity in video
data~\cite{Tralie2018b}, chatter detection in turning
processes~\cite{Khasawneh2018a, Yesilli2019}, stability of delayed
equations~\cite{Khasawneh2017}, and complex network analysis~\cite{Myers2019}.
While the persistent homology of point clouds has been shown to be a promising
tool for time series analysis, in this work, we take the first approach and
study the statistical
analysis of persistence of the sublevel sets of the original time
series.

Through persistence, we can capture the significance of sublevel sets in the time series, which can be extended to higher dimensional data-sets (a more detailed introduction follows in
\secref{sublevel_set_persistence_background}).
Sublevel set persistence has been successfully used for a wide range of time series analysis applications including local extrema detection~\cite{Gholizadeh2018}, true step detection in signal processing~\cite{Khasawneh2018b}, fourier spectrum analysis~\cite{Myers2020}, arrythmia detection~\cite{Dindin2020}, and cancer histology~\cite{Lawson2019}.

One of the most attractive features of sublevel set persistence as a time series
analysis tool is its stability to noise perturbations
(see, e.g.,~\cite{Cohen-Steiner2006}). However, even with the stability of
persistence diagrams, there can be many points in the diagram which we wish to distinguish as associated to either noise or true signal. One method of doing this is by estimating a confidence interval for each point in the persistence diagram, which has recently been of interest in the field of TDA~\cite{Bubenik2015, Chazal2013, Fasy2014}. 
While methods based on bootstrap resamping~\cite{chazal2017robust, Chazal2013, Fasy2014} and entropy~\cite{Atienza2017} have been developed for separating noise from signal in the persistence diagrams, they were not strictly designed for sublevel set persistence. 
This can lead to problems in both implementation and computational demand when applied to sublevel set persistence. These issues are discussed in Section~\ref{sec:sublevel_set_persistence_background} and then demonstrated in Section~\ref{ssec:comparison}.

This manuscript is organized as follows. In
\secref{sublevel_set_persistence_background}, we first provide an
introduction of persistence for time series data and the effects of additive noise on
the persistence diagram as well as commonly used methods for seperating noise from signal. In \secref{stats_of_noise}, we introduce a novel
analysis of the statistics of the lifetimes of the sublevel set persistence diagram. In
\secref{noise_models_applied}, we realize the statistical analysis for
cutoff expressions for Gaussian, uniform, Rayleigh, and exponential additive
noise distribution. In \secref{stats_param_estimation} and
\ref{sec:signal_compensation}, we provide a method to approximate the
distribution parameter of the additive noise as well as a signal compensation
term, which is required in the cutoff equation. Lastly, in
\secref{examples}, we apply our ANAPT method to both numeric and experimental time
series to validate its functionality and make comparisons to existing methods.


\section{Sublevel Set Persistence} \label{sec:sublevel_set_persistence_background}
We now provide an introduction to persistence, as it applies to computing
persistence of the sublevel sets of a time series. Let us begin with the single
variable function $f:\mathbb{R} \rightarrow \mathbb{R}$.
Given~$r \in \R$, we define the \emph{sublevel set
below~$r$} as~$f^{-1}(-\infty,r]$. As the filtration parameter
$r$ increases, the sublevel sets may grow but remain the same (up to homology) until a local extrema (i.e., a
local minimum or
maximum) is reached.\footnote{Here, we assume that the function $f$ satisfies
some mild ``niceness'' conditions, such as being
q-tame~\cite{chazal2016structure}.}
If the extrema is a local minimum, then a new connected component or set is
``born'' at height~$r_B$; we label
that set with the value $r_B$. On the
other hand, if the extrema is a local maximum, two previously existing sets are
combined.  If the two sets were labeled~$r_B$ and~$r_B'$, with~$r_B \leq
r_B'$ and the maximum attained at $r_D$, then, by the Elder Rule~\cite[p.~150]{Edelsbrunner2010}, we say that
the component born at $r_B'$ dies going into~$r_D$, and the resulting set
assumes the label $r_B$.  The pair
$(r_B', r_D)$ is called a persistence pair. As $r$ ranges from~$-\infty$ to~$\infty$, the
\emph{persistence diagram} is the collection of all $n$ such pairs,~$\dgm{f}=\{(b_i,d_i)\}_{i=1}^n$.  Any unpaired births are called
\emph{essential classes} and are paired with a death coordinate of~$\infty$;
thus,~$\dgm{f}$ is embedded in the \emph{extended plane} $\overline{\R}^2$.
The \emph{lifetime} or \emph{persistence} of a point~$(b_i,d_i) \in \dgm{f}$ is defined as~$\ell_i = d_i-b_i$.
In this paper, our functions are only sampled on a finite domain, with the first
sample at time $t_a$ and the last sample at time $t_b$.  We obtain a
continuous function over $[t_a,t_b]$ by using a
piecewise linear interpolation between consecutive samples, and extending the
function to $\pm \infty$ by extending the first (resp., last) edges to rays.
Doing so allows us to define a persistence diagram that does not have
critical points on the boundary of our time series.  As such, we study the
persistence points where both coordinates are finite, and omit persistence
points that contain an unbounded coordinate.

To give a concrete example of a persistence diagram, we demonstrate a
simple example for the function shown in \figref{0d_sublevel_persistence}.
This function has thirteen sample points, two local minima, and two local
maxima.
\begin{figure}[h]
    \centering
    \includegraphics[width = 0.55\textwidth]{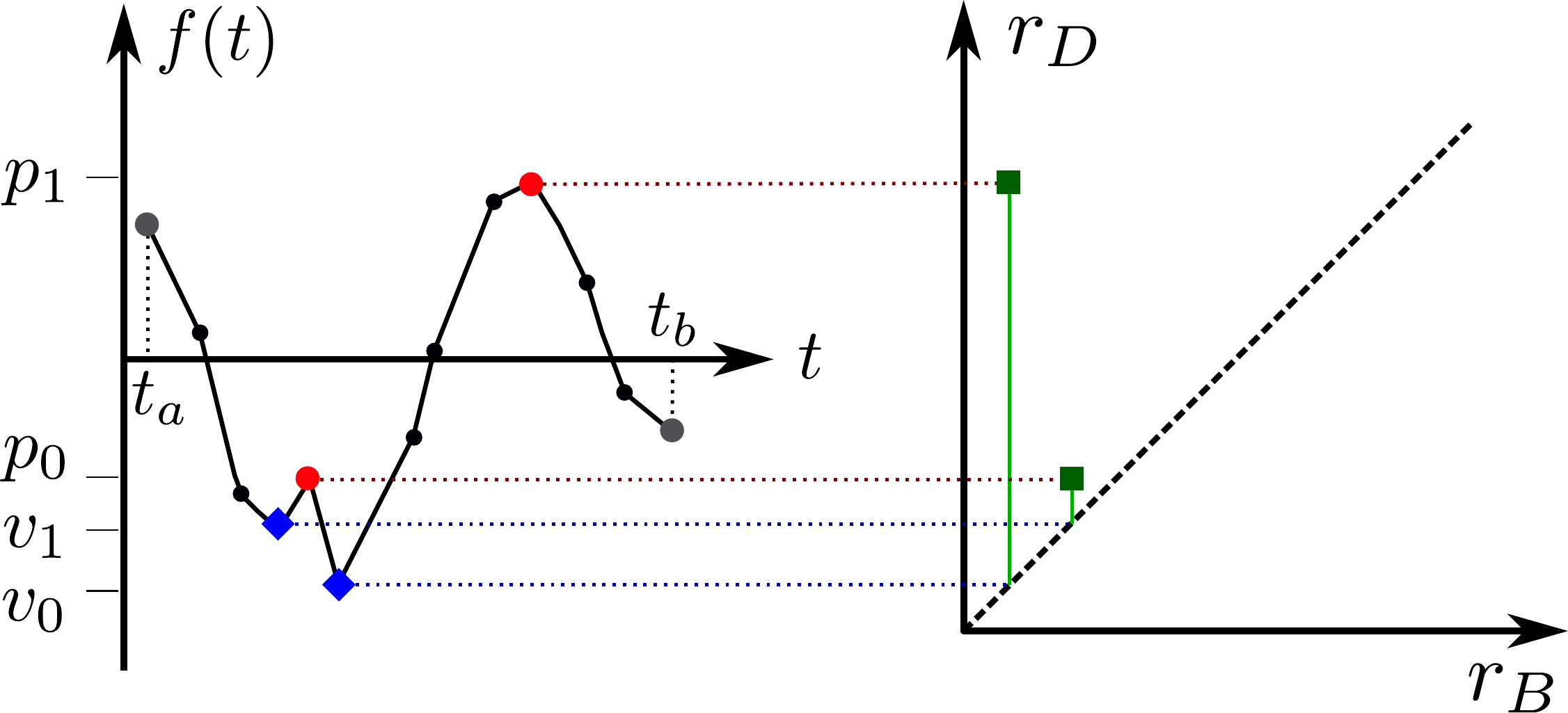}
    \caption{Persistence diagram summarizing the sublevel sets of
    a function $f(t)$ over finite domain $t \in[t_a, t_b]$.
    This function has two local minima (blue squares) and two local maxima (red
    circles.
    }
    \label{fig:0d_sublevel_persistence}
\end{figure}
The lowest finite critical value of the function occurs at height~$v_0$.  For all
$r<v_0$, $f^{-1}(-\infty,r]$ is the ray $[f^{-1}(r),\infty)$, where we observe
that $f^{-1}(r)$ is a well-defined and greater than $t_b$.  This connected
component is labeled with $-\infty$, since it is ``born'' at $-\infty$.
Then, at height~$r=v_0$, a second connected component is born.
The next topological change occurs at height $r=v_1$, where a third connected component is born.
The next
extrema is reached when~$r = p_0$. At this extrema, the sublevel set that
was born at $r = v_1$ dies, while the sublevel set born at $r = v_0$
persists based on the Elder Rule. This pair $(v_1,p_0)$ is recorded in the
persistence diagram. From here, the next change happens at $r =
p_1$, where the second sublevel set dies and is recorded in the persistence
diagram as~$(v_0,p_1)$.
Then, no further topological changes occur, but this sublevel set continues to
grow as $r$ grows.  This essential class is recorded in the persistence diagram as
$(-\infty,\infty)$ and is not studied in our analysis.
As shown in the persistence diagram, the point~$(v_1, p_0)$ is close to
the diagonal (the line $y=x$), which signifies that the sublevel set only persisted for a short
range of heights ($r$); on the other hand, the point~$(v_0, p_1)$ is far from
the diagonal, suggesting it was from a significant sublevel set.
The algorithm used to calculate the this persistence diagram is detailed in \appref{algorithms}.

The idea of persistence can be extended to higher dimensions allowing for the
analysis of the shape of high-dimensional data sets. However, for our work, we
only need to analyze the zero-dimensional features (i.e., connected components)
of a one-dimensional function. A more thorough
background on TDA, and persistent homology specifically, can be found
in~\cite{Edelsbrunner2008, Munch2017, Perea2018}. Other common ways for studying time series with a similar perspective is through merge trees or dendograms~\cite{Carlsson2012, Chowdhury2017, Lee2012}.

\mypara{Sublevel Set Persistence with Additive Noise} \label{ssec:noise_effects}

We now investigate the stability of sublevel set persistence diagrams to additive noise for single variable functions. To illustrate the stability,
we first take an example time series with additive noise as $x(t) +
\epsilon$, where $x(t)$ is sampled at a uniform rate $f_s$ and $\epsilon$ is
additive noise from the noise model~$\mathcal{N}$. An example of a persistence diagram from the time series with
additive noise $\dgm{x + \epsilon}$ is shown in
\figref{stats_overview_fig}, along with the diagram without the additive
noise~$\dgm{x}$. This example also demonstrates how a cutoff
$C_\alpha$ can be used to separate the significant points in the persistence diagram and those associated to the additive noise.
\begin{figure}[h]
    \centering
    \includegraphics[scale = 0.52]{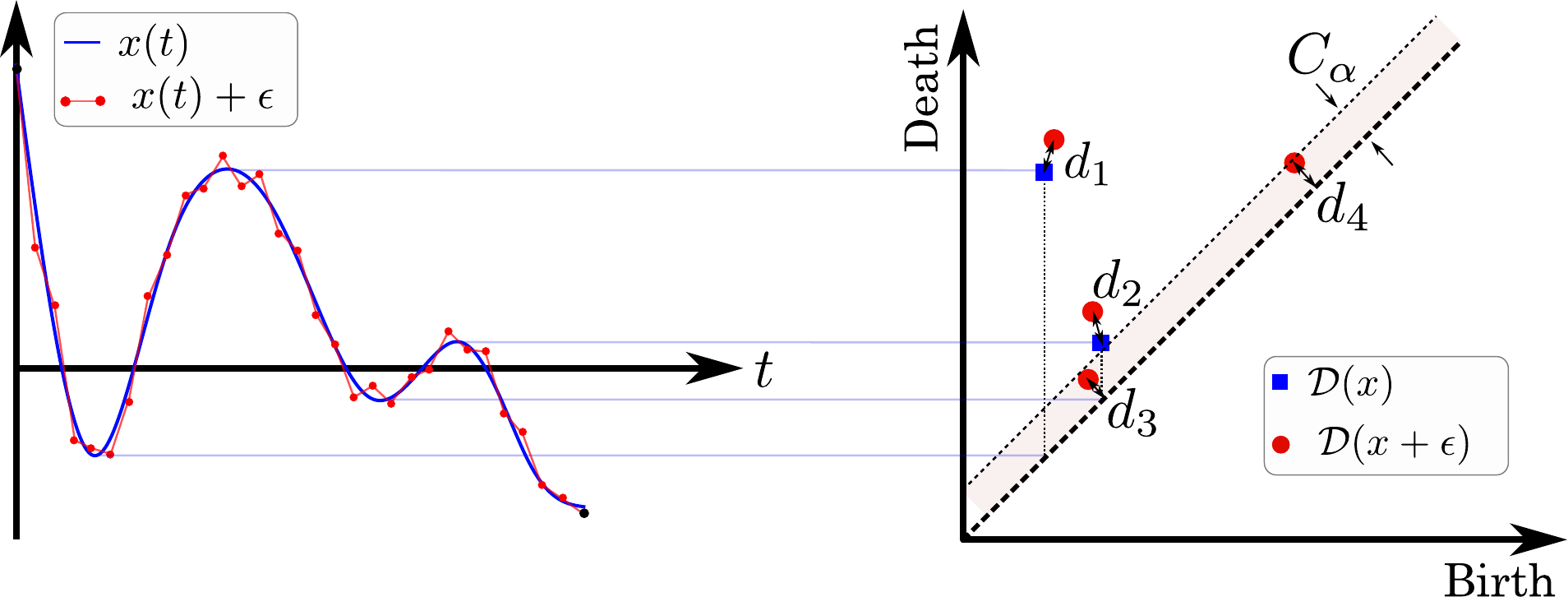}
    \caption{Sublevel set persistence applied to $x(t)$ of a single variable
    function or time series with and without additive noise $\epsilon$ from $\mathcal{N}$, shown
    in red and blue, respectively. This
    demonstrates the stability of persistent homology with the time series
    (left) with and without additive noise and the small effect on the resulting
    persistence diagrams (right). In addition, the light red region separates the significant features from those associated to additive noise.}
    \label{fig:stats_overview_fig}
\end{figure}

This example demonstrates that the addition of (small) noise does not have a large
effect on the position of significant sublevel sets in the persistence diagram
with the distances between significant points ($d_1$ and $d_2$) all being
relatively small. This is no surprise due to the stability theorem of
the bottleneck distance for persistence diagrams~\cite{Cohen-Steiner2006,chazal2016structure},
where the bottleneck distance is defined as the minimum distance to match two persistence diagrams.
For example, if we assume $d_1>d_2>d_3>d_4$, then the bottleneck distance would
be $d_1$. However, additive noise does introduce several points in the persistence diagram located near the diagonal with relatively small lifetimes.
These noise-artifact
persistence pairs are formed from the peak-valley pairs in the additive noise. This work
focuses on a statistical analysis of these lifetimes to develop a method for
separating the significant persistence diagram points from those of additive
noise, shown in as light red region in example persistence diagram of
\figref{stats_overview_fig}, through a cutoff $C_\alpha$ with~$\alpha \in [0,1] $ as the given confidence level.

\mypara{Current Methods for Separating Noise from Signal in the Persistence Diagram}
The methods for developing confidence sets and separating noise from signal for persistence diagrams using
resampling techniques~\cite{Chazal2013, Fasy2014} were not strictly developed
for persistence of sublevel sets of real-valued time series. This means that using
the bottleneck bootstrap as presented in \cite{chazal2017robust} gives
conservative confidence sets. 
One solution to this problem could be to tailor these methods for sublevel set persistence by applying resampling techniques for time series (e.g., stationary bootstrap~\cite{Politis1994} and the sieve bootstrap~\cite{Buehlmann1997}). 
However, these methods have general requirements of weakly dependent stationary observations or exceptionally high sampling rates such that $f_s >> f'$, where $f'$ the highest dominant frequency of the time series. 
An alternative approach to applying a bootstrap resampling is to use a function fitting procedure that is suitable for strongly dependent time series. Using the function fit of the signal, the residuals can be resampled to the time series to apply the bottleneck bootstrap technique. 
Unfortunately, this method is plagued with additionally computational steps including the function fitting initial step, repeated sublevel set persistence calculations, and the bottleneck distance calculation. The possibility of signal distortion when fitting to the curve can also cause for inaccurate cutoff estimations.
Another method for separating noise from significant features in a persistence
diagram is persistent entropy~\cite{Atienza2017}. However, this method may not
properly distinguish between noise and significant features in the persistence
diagram if the number of significant points in the persistence diagram is
relatively large compared to the amount of noise.

Because of the significant drawbacks of the current methods for developing
confidence sets and associated cutoffs in the persistence diagram for sublevel
set persistence, we have developed a new method that directly analyzes the statistics of additive noise in the sublevel set persistence diagram. 
We have also made the Python code for implementing the $\Theta(n\log(n))$ algorithm (see Section~\ref{app:algorithms}) discussed throughout the manuscript publicly available through the Python package \texttt{teaspoon}\footnote{\label{footnote_label}\url{https://lizliz.github.io/teaspoon/}}.


\section{Statistics of Additive Noise in the Persistence Diagram}
\label{sec:stats_of_noise}
Before studying a time series with additive noise, $x+ \epsilon \colon \R \to
\R$,
we analyze the statistics of sublevel set persistence diagrams of
the noise alone.
Our goal is to leverage this analysis in order to generate a cutoff in the persistence diagram
to separate out these noise-artifact points in the persistence diagram for
$\dgm{x+\epsilon}$ from the points that capture true features of $x$.

\subsection{Sublevel Set Persistence Diagram Statistics Background}
\label{ssec:stats_relationship_background}
We start with the noise, which can be thought of as a (sampled) function
$\epsilon \colon \R \to \R$, where, for each $t \in \R$, the value $\epsilon(t)$ is a
random independently and identically distributed (iid) variable sampled from
a predefined noise distribution $\mathcal{N}$.
In our noise model, there is no covariance structure between these~$\R$-indexed
family of random variables.
Let $\noise$
denote the induced probability distribution over functions from~$\R$ to $\R$, so~$\epsilon \sim \noise$.

The first step in developing a cutoff based on the persistence diagram
statistics of additive noise $\dgm{\epsilon}$ is to determine a
relationship between the descriptive additive noise distribution parameters and
the distribution of the lifetimes. To do
this, we develop an expression for the expected lifetime of points in
$\dgm{\epsilon}$.
Let~$f \colon \R \to \R$ and~$F \colon \R \to \R$ be the
probability density function and cumulative density function of $\mathcal{N}$, respectively.
Let~$f_B \colon \R \to \R$ and~$f_D \colon \R \to \R$ be the
probability density functions
for the local minima (corresponding to births) and maxima (corresponding to
deaths) of the sublevel sets
from a random noise function $\varepsilon \sim \noise$.  Let~$F_B$ and~$F_D$ be the corresponding
cumulative density functions.
By the commutative property of addition and the
definition of a lifetime being the difference between the death and birth times,
the expected or mean lifetime $\mu_L$ is the difference between
the expected birth times~$\mu_B := \mathbb{E}(B)$ and death times~$\mu_D :=
\mathbb{E}(D)$, where $B$ and $D$ are the sets of birth and death values,~as
\begin{equation}\label{eq:muL_expression}
    \mu_L  := \mu_D - \mu_B
           = \int_{-\infty}^{\infty} z \left[ f_D(z)-f_B(z) \right] dz.
\end{equation}
A formal proof of this relationship is
provided in \thmref{mean-lifetime} of the appendix.
From \eqref{muL_expression}, we move forward knowing that $\mu_L$ can be
defined using only expressions for~$f_B$ and~$f_D$.
In other words, we only need the distribution of birth and death times, not of
the lifetimes, which would require knowing how the births and deaths are paired.

Let $z \in \R$ and $\epsilon_L, \epsilon_R \stackrel{iid}{\sim} \mathcal{N}$.
Because~$\epsilon_{L}$ and $\epsilon_{R}$ are independent, by the multiplication
rule for probability, we can state that the probability
that~$\max\{ \epsilon_L, \epsilon_R \} < z$ is
\begin{equation}
    p(\epsilon_{L} < z) p(\epsilon_{R}<z).
    \label{eq:maxima_prob_full}
\end{equation}
Recalling that $F$ is the cumulative distribution function of $\mathcal{N}$,
we can equate these probabilities as
$$
    p(\epsilon_{L}<z)
        = p(\epsilon_{R}<z)
        = F(z).
$$
Combining with
\eqref{maxima_prob_full}, we define the probability distribution function of
maximas~$f_{\max} \colon \R \to \R$ by
\begin{equation}
    f_{\max}(z) = F^2(z).
\label{eq:maxima_prob_dist}
\end{equation}
Similarly, we can show that the local minima probability distribution function
$f_{\min} \colon \R \to \R$ is defined by
\begin{equation}
    f_{\min}(z) = [1-F(z)]^2,
\label{eq:minima_prob_dist}
\end{equation}
since $p(\epsilon_{L}>z) = p(\epsilon_{R}>z) = 1 - F(z)$.

Using these distributions,
we calculate the local maxima probability density function~$f_D$, which
represents the distribution of death times.
We proceed in the discrete setting and consider $\epsilon \sim \noise$ represented by a
finite sequence of values; that is, we have~$\epsilon = ( \epsilon_1, \epsilon_2, \ldots, \epsilon_n )$ where $\epsilon_i
\stackrel{iid}{\sim} \mathcal{N}$. Fixing $z=\epsilon_i$
$\epsilon_L=\epsilon_{i-1}$, and~$\epsilon_R=\epsilon_{i+1}$, we know the
probability that $\epsilon_i$ is a maximum (or a minimum) from above.
To calculate $f_D$ (resp., $f_B$), we take the convolution of $f$ and $f_{\max}$
(resp., of $f$ and $f_{\min}$) and
normalize (to ensure the density integrates to unity):
\begin{equation}
    \begin{split}
        {f}_D(z) & := \frac{f(z)F^2(z)}{N_D},\\
        {f}_B(z) & := \frac{f(z)[1-F(z)]^2}{N_B},
    \end{split}
    \label{eq:f_proper}
\end{equation}
where $N_B = \int_{-\infty}^{\infty}f(z)[1-F(z)]^2dz$ and $N_D = \int_{-\infty}^{\infty}f(z)F^2(x)dz$.
We calculate $N_B$ and $N_D$ from \eqref{f_proper} by substituting $f(z)
= F'(z)$, which reduces the $N_D$ equation to
\begin{equation}
\begin{split}
N_D & = \int_{-\infty}^{\infty} F'(z)F^2(z)dz = \int_{-\infty}^{\infty} \frac{1}{3}{(F^3(z)})'dz \\
    & = \frac{1}{3} F^3(z)|_{-\infty}^{\infty}  = \frac{1}{3},
\end{split}
\label{eq:N_D}
\end{equation}
since $F(\infty) = 1$ and $F(-\infty) = 0$. Similarly,
\begin{equation}
\begin{split}
N_B & = \int_{-\infty}^{\infty} f(z)[1-F(z)]^2 dz  = \int_{-\infty}^{\infty} f(z)[1-2F(z) + F^2(z)] dz  = N_D + \int_{-\infty}^{\infty} f(z)[1-2F(z)] dx \\
	& = N_D + \int_{-\infty}^{\infty} F'(z) dz - \int_{-\infty}^{\infty} (F^2(z))' dz = N_D + \left[ F(z) - F^2(z) \right]|_{-\infty}^{\infty} = N_D = \frac{1}{3}.
\end{split}
\label{eq:N_B}
\end{equation}
We can now reduce \eqref{f_proper} to
\begin{equation}
\begin{split}
{f}_B(z) & = 3{f(z)[1-F(z)]^2}, \\
{f}_D(z) & = 3{f(z)F^2(z)},
\end{split}
\label{eq:f_proper_simplified}
\end{equation}

For demonstrative purposes we now assume $f(z)$ is of a Gaussian distribution to validate our expressions in \eqref{f_proper}.
Specifically, the Gaussian probability density function is defined as
\begin{equation}
f(z) = \frac{1}{\sqrt{2\pi \sigma^2}}e^{-\frac{(z-\mu)^2}{2\sigma^2}},
\label{eq:f_norm}
\end{equation}
with a cumulative distribution
\begin{equation}
F(z) = \frac{1}{2}\left[  1+{\rm erf}\left(  \frac{z-\mu}{\sigma \sqrt{2}} \right)  \right].
\label{eq:F_norm}
\end{equation}
To validate the resulting expressions for ${f}_B(z)$ and ${f}_D(z)$ in
\eqref{f_proper_simplified}, a numerical simulation of a normal distribution
$\mathcal{N}_n(\mu = 0, \sigma^2 = 1)$ of length $n = 10E5$ is used (see
\figref{histos_of_B_D_th_comp}). This analysis shows a very similar result
between the histograms $h(*)$ and distributions.
\begin{figure}[h]
    \centering
    \includegraphics[scale = 0.55]{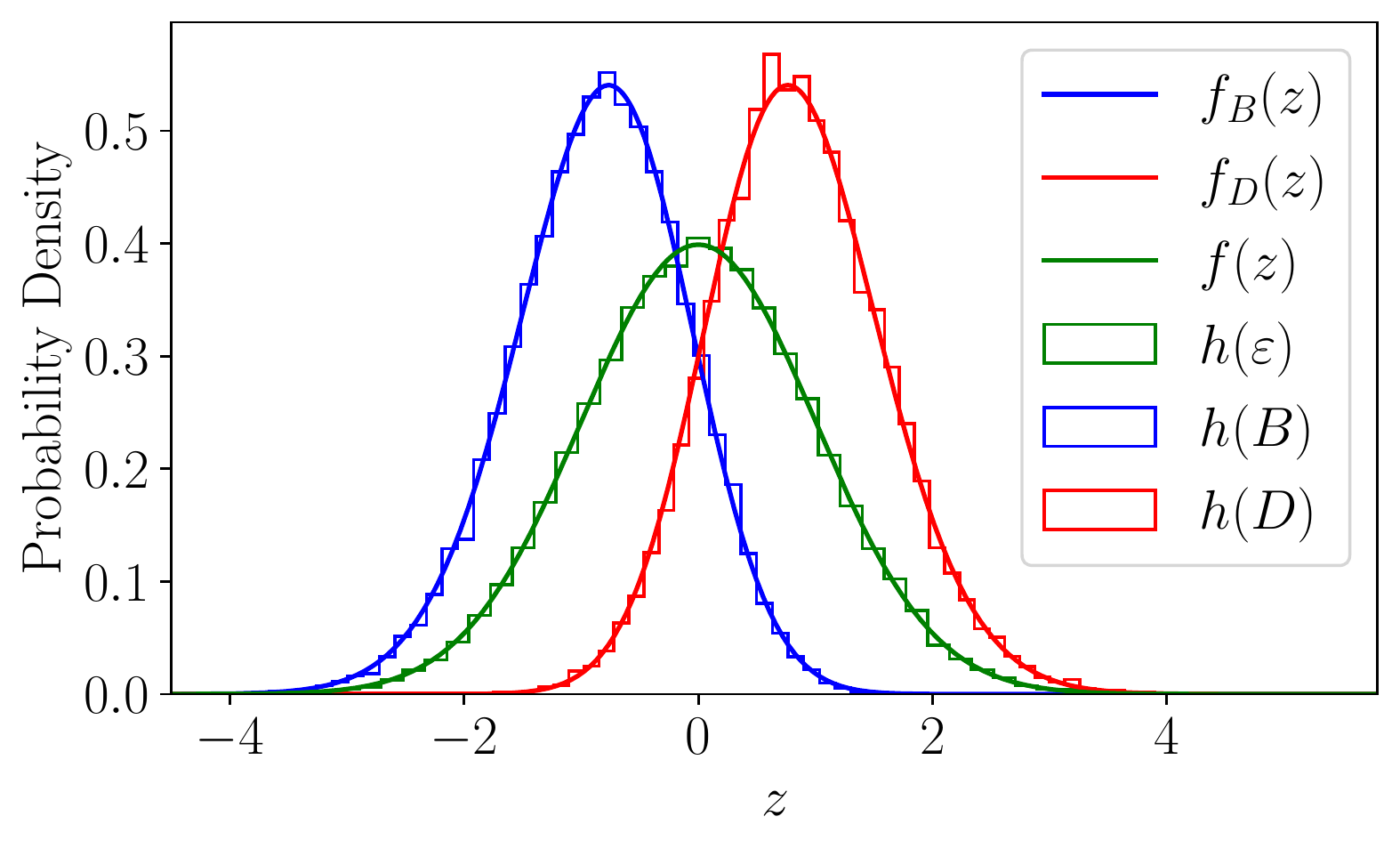}
    \caption{Histograms $h(*)$ of the zero mean normal distriubtion $\mathcal{N}(0,\sigma^2 =1)$ and the resulting birth times $B$ and death times $D$, which are compared to the density distributions from \eqref{f_proper}.}
    \label{fig:histos_of_B_D_th_comp}
\end{figure}

Now that we have shown that our expressions for the probability distribution of
the minima and maxima are correct, we proceed to correlate the mean lifetime
$\mu_L$ to the additive noise distribution parameters. From our results in
\eqref{f_proper_simplified} applied to \eqref{muL_expression} we calculate the mean lifetime as
\begin{equation}
\mu_L = 3\int_{-\infty}^{\infty} z f(z)\left[ F^2(z)-(1-F(z))^2 \right] dz =      3\int_{-\infty}^{\infty} z \left[ (F^2(z))'-F'(z) \right] dz,
\label{eq:muL_expression_simplified}
\end{equation}
which is simplified using integration by parts as
\begin{equation}
\boxed{\mu_L = 3\int_{-\infty}^{\infty} F(z)\left[ 1-F(z) \right] dz.}
\label{eq:muL_expression_final}
\end{equation}

From the numerical simulation in \figref{histos_of_B_D_th_comp} we also found the $\bar{L} \approx 1.6921$, where $\bar{L}$ is the sample mean of the sublevel set persistence lifetimes from $\mathcal{N}(0,\sigma^2 =1)$. This is very similar compared  to the numerically estimated theoretical $\mu_L \approx 1.6925$ from \eqref{muL_expression_final}. These results suggest that \eqref{muL_expression_final} is correct.

\subsection{Cutoff Background} \label{ssec:extrema_distributions_for_cutoff}
To determine a suitable cutoff using our ANAPT method, we again start by assuming we have $n$ random
samples from our noise distribution:~${\epsilon} = \{ \epsilon_1, \epsilon_2, \ldots, \epsilon_n \} \stackrel{iid}{\sim} \mathcal{N}$ with a cumulative
probability function $F$. The probability that the minimum of~${\bf \epsilon}$ is
less than the value $a$ is equivalent to
\begin{equation}
P(\min({\bf \epsilon}) < a) = 1 - P(\epsilon_1>a, \epsilon_2>a,\ldots, \epsilon_n>a),
\label{eq:P_min_generic}
\end{equation}
where $P(\epsilon_i>a) = 1 - F(a)$. If we extend this relationship to all $n$
realizations, we can express this probability as
\begin{equation}
P(\min({\bf \epsilon}) < a) = 1 - (1-F(a))^n.
\label{eq:P_min_generic}
\end{equation}
Similarly, for $b>a$, an expression for the probability that an element of ${ \epsilon}$ is
greater than~$b$~is
\begin{equation}
P\left(\max({\epsilon}) > b\right) = 1 - (F(b))^n.
\label{eq:P_max_generic}
\end{equation}
If we now take both of these probabilities, we can extend them to the maximum
finite lifetime as $\max(L) \lessapprox \max({\epsilon}) -\min({\epsilon})$. we
can use this to generate a probability of a lifetime being greater than $b-a$ as
\begin{equation}
\alpha = P(\max(L) > b-a) \gtrapprox P\big(\max({\epsilon}) > b, \min({\epsilon}) < a\big) = (1 - [F(b)]^n)(1 - [1-F(a)]^n),
\label{eq:L_max_generic}
\end{equation}
where $\alpha$ is the confidence of this event occurring.  If the $f(z)$
associated to $F(z)$ of \eqref{L_max_generic} is symmetric about some mean $\mu$
such that $c = b-\mu = \mu-a$, we can reduce \eqref{L_max_generic} to
\begin{equation}
\alpha = {(1 - [F(c)]^n)}^2
\label{eq:L_max_generic_symmetric}
\end{equation}
since $F(b)= 1-F(a)$ for the symmetric case.
\eqref{L_max_generic_symmetric} can be then solved for $c$ as
\begin{equation}
c = F^{-1}\left[\left(1-\sqrt{\alpha}\right)^{1/n}\right].
\label{eq:L_max_generic_symmetric2}
\end{equation}
Additionally, we know that a cutoff should be set such that $C_\alpha = b-a =
2c$ resulting in a cutoff equation as
\begin{equation}
\boxed{C_\alpha = 2F^{-1}\left[{\left(1-\sqrt{\alpha}\right)}^{1/n}\right].}
\label{eq:cutoff_symmetric}
\end{equation}

If there is no symmetry in the distribution then we need a new cutoff equation.
To do this, we return to our probability equation as
\begin{equation}
\alpha = P(\max(L) > b-a) \gtrapprox P(\min({\epsilon}) < a, \max({\epsilon}) >b)  = (1-{[1-F(a)]}^n)(1 - [F(b)]^n),
\label{eq:alpha_exact_non_symmetric}
\end{equation}
However, unlike \eqref{cutoff_symmetric}, we can not solve \eqref{L_max_generic}
for a parameter $c$ due to their being no symmetry between $a$ and $b$ about a
mean $\mu$ which means we must simplify \eqref{alpha_exact_non_symmetric}.  To do this, we assume that $P(\min({\epsilon}) < a) = P(\max({\epsilon})
>b)$ or $1-{[1-F(a)]}^n = 1 - [F(b)]^n = \sqrt{\alpha}$. We then solve for
$a$ and~$b$ separately as
\begin{equation}
a = F^{-1}\left[  1 - {(1-\sqrt{\alpha})}^{1/n}  \right]
\label{eq:prob_a}
\end{equation}
and
\begin{equation}
b = F^{-1}\left[ {(1-\sqrt{\alpha})}^{1/n}  \right].
\label{eq:prob_b}
\end{equation}
With $C_\alpha = b-a$ and the values of $a$ and $b$ from
\eqref{prob_a} and \eqref{prob_b}, respectively, we can solve for our general
cutoff expression as
\begin{equation}
\boxed{C_\alpha = F^{-1}\left[ {(1-\sqrt{\alpha})}^{1/n}  \right] - F^{-1}\left[  1 - {(1-\sqrt{\alpha})}^{1/n}  \right] .}
\label{eq:cutoff_nonsymmetric}
\end{equation}

For our application we want to have a high confidence level that no persistence pairs associated to noise occur with a lifetime greater than the cutoff. We suggest
using a confidence level of~$\alpha = 0.001$ or $0.1\%$ for a low probability of this occurring.



As a summary, the ANAPT cutoff equations in \eqref{cutoff_symmetric}~and~\eqref{cutoff_nonsymmetric} are only
dependent on the desired confidence $\alpha$, the signal length $n$ and the
cumulative probability distribution $F(z)$. In \secref{noise_models_applied} we demonstrate how to apply \eqref{cutoff_symmetric} and \eqref{cutoff_nonsymmetric} for the Gaussian, uniform, Rayleigh, and exponential distribution.


\section{Cutoff for Noise Models} \label{sec:noise_models_applied}
For applying noise models to the confidence levels in
\eqreftwo{L_max_generic}{L_max_generic_symmetric} for our ANAPT method, we need to be
either given the additive noise parameters, or estimate them from the lifetimes.
However, before this can be done, we need to understand which parameters are
needed given the additive noise distribution $f(z)$. We do this analysis
for Gaussian (normal), Uniform, Rayleigh, and exponential
distributions as shown in \figref{distributions_of_noise_models}.
\begin{figure}[h]
    \centering
    \includegraphics[scale = 0.37]{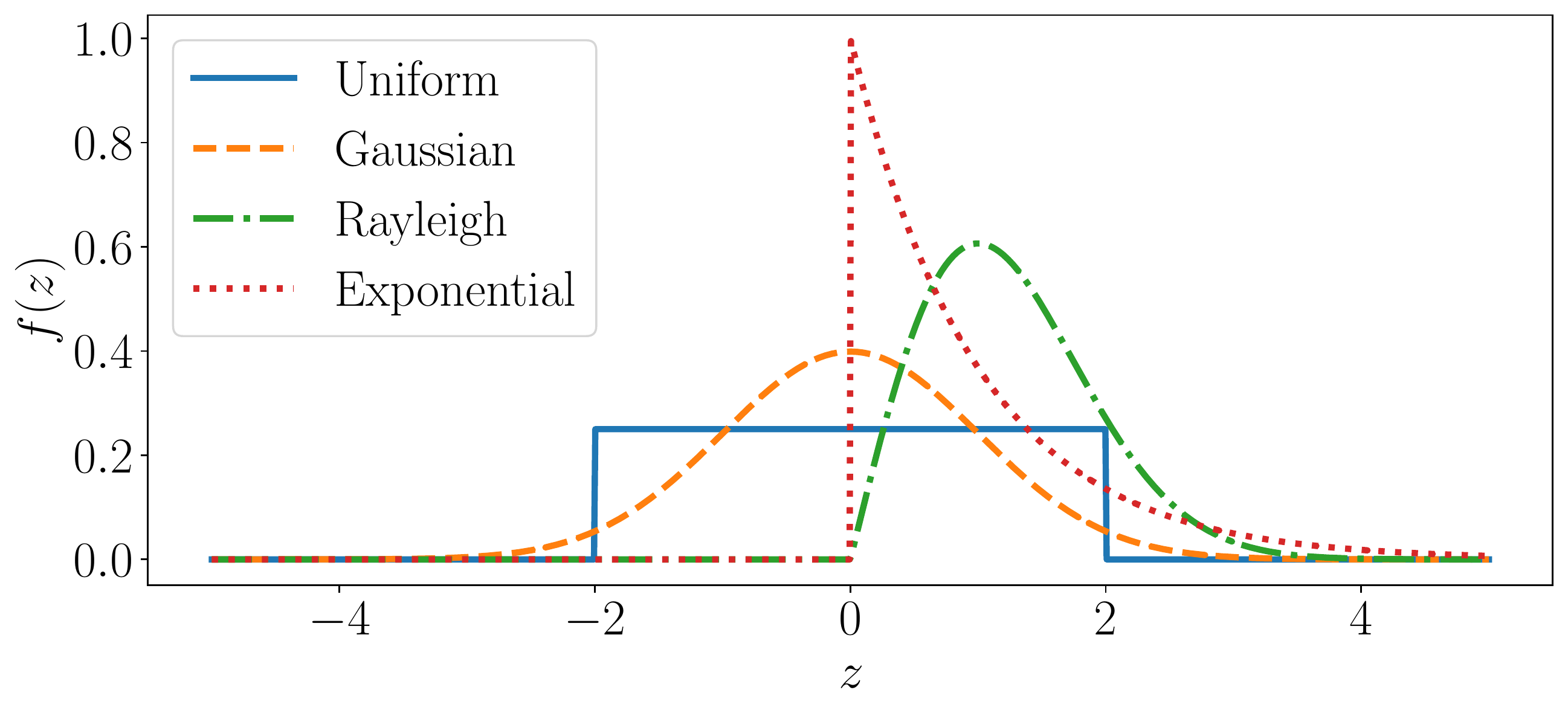}
    \caption{Additive noise probability distributions $f(z)$ for the four models realized in this work: uniform, Gaussian, Rayleigh, and exponential.}
    \label{fig:distributions_of_noise_models}
\end{figure}

\subsection{Cutoff for Gaussian Noise}
We start our analysis with the commonly used Gaussian distribution model. The Gaussian (normal) probability distribution function is defined as
\begin{equation}
f(z) = \frac{1}{\sqrt{2\pi \sigma^2}}e^{-\frac{(z-\mu)^2}{2\sigma^2}},
\label{eq:gaus_f_norm}
\end{equation}
with a cumulative distribution function
\begin{equation}
F(z) = \frac{1}{2}\left[  1+{\rm erf}\left(  \frac{z-\mu}{\sigma \sqrt{2}} \right)  \right].
\label{eq:gaus_F_norm}
\end{equation}
We start by solving for the inverse of \eqref{gaus_F_norm} as
\begin{equation}
F^{-1}(u) = \sqrt{2} \sigma {\rm erf}^{-1}\left(2u - 1\right) + \mu.
\label{eq:gaus_F_norm_inv}
\end{equation}
Since the mean shift $\mu$ has no effect on the sublevel set lifetimes we can
ignore it and apply \eqref{gaus_F_norm_inv} with~$\mu = 0$ to solve for the cutoff from \eqref{cutoff_symmetric} as
\begin{equation}
C_\alpha = 2^{3/2}\sigma \: {\rm erf}^{-1}\left[2(1-\sqrt{\alpha})^{1/n} - 1\right].
\label{eq:cutoff_normal}
\end{equation}
With a full development of the statistics of sublevel set persistence for
Gaussian (normal) additive noise we are able to determine a suitable cutoff for
iid noise with only the distribution parameter $\sigma$ needing estimated.

\subsection{Cutoff for Uniform Noise}
Let $a < b \in \R$.
The uniform distribution over the interval~$[a,b]$ has a probability density function defined as
\begin{equation}
f(z)=
\left\{ \begin{array}{ll}
   \frac{1}{b-a} & z\in[a,b] \\
   0 &\text{otherwise}
\end{array} \right.
\label{eq:unif_F_norm}
\end{equation}
with a cumulative distribution function
\begin{equation}
F(x)=
\left\{ \begin{array}{ll}
   0 & z<a \\
   \frac{z-a}{b-a} & z\in[a,b] \\
   1 & z>b.
\end{array} \right.
\label{eq:unif_F_norm}
\end{equation}
By assuming a symmetric distribution about zero (this assumptions
does not influence the resulting cutoff due to the properties sublevel set
persistence lifetime) such that $a = -b$ and~$\Delta = b-a$.
This changes $F(x)$ to
\begin{equation}
F(z)=
\left\{ \begin{array}{ll}
   0 & z<-\frac{\Delta}{2} \\
   \frac{2z+\Delta}{2\Delta} & z\in[-\frac{\Delta}{2},\frac{\Delta}{2}] \\
   1 & z>\frac{\Delta}{2}
\end{array} \right.
\label{eq:unif_F_norm_changed}
\end{equation}
If we now apply \eqref{cutoff_symmetric} to the inverse of the cumulative probability distribution in \eqref{unif_F_norm_changed}, we can calculate $C_\alpha$ as
\begin{equation}
C_\alpha = {\Delta}\left[ {2\left(1-\sqrt{\alpha}\right)}^{1/n} - 1 \right].
\label{eq:cutoff_uniform}
\end{equation}
Equation~(\ref{eq:cutoff_uniform}) only requires an estimate of the distribution parameter $\Delta$ as both $\alpha$ and $n$ are chosen as desired and the length of the time series, respectively.

\subsection{Cutoff for Rayleigh Noise}
The Rayleigh distribution has a probability density function over the domain $x \in [0, \infty)$ and is defined as
\begin{equation}
f(z) = \frac{z}{\sigma^2}e^{-\frac{z^2}{2\sigma^2}},
\label{eq:rayleigh_f_norm}
\end{equation}
with a cumulative distribution function
\begin{equation}
F(z) = 1-e^{-\frac{z^2}{2\sigma^2}}.
\label{eq:rayleigh_F_norm}
\end{equation}
Since this distribution is asymmetric we use \eqref{cutoff_nonsymmetric} to calculate $C_\alpha$ as
\begin{equation}
C_\alpha = \sigma \left( \sqrt{-2 \ln\left([1-\sqrt{\alpha}]^{1/n}\right)} - \sqrt{-2 \ln\left(1-[1-\sqrt{\alpha}]^{1/n}\right)}  \right),
\label{eq:cutoff_rayleigh}
\end{equation}
where $\sigma$ is the only parameter that needs to be provided to calculate the cutoff.

\subsection{Cutoff for Exponential Noise}
The exponential distribution has a probability density function over the domain $z \in [0, \infty)$ and is defined as
\begin{equation}
f(z) = \lambda e^{-\lambda z},
\label{eq:exp_f_norm}
\end{equation}
with a cumulative distribution function
\begin{equation}
F(z) = 1- e^{-\lambda z},
\label{eq:exp_F_norm}
\end{equation}
where $\lambda$ is the distribution parameter with $\lambda > 0$. This this distribution is also asymmetric, so we use \eqref{cutoff_nonsymmetric} to calculate $C_\alpha$ as
\begin{equation}
C_\alpha = -\frac{1}{\lambda} \ln \left(  [1-\sqrt{\alpha}]^{1/n} - [1-\sqrt{\alpha}]^{2/n}  \right),
\label{eq:cutoff_exp}
\end{equation}
where $\lambda$ is the only parameter that needs to be provided to calculate the cutoff.


\section{Cutoff and Distribution Parameter Estimation Method} \label{sec:stats_param_estimation}

If the distribution parameter is know ($\sigma$ for Gaussian distributions, $\Delta$ for uniform distributions, $\sigma$ for Rayleigh distributions, and $\lambda$ for exponential distributions), then the cutoff $C_\alpha$ using our ANAPT method can be
calculated simply with the use of the correct cutoff equation in
\secref{noise_models_applied} and the subsequent analysis may be
skipped. However, in most real-world time series it is uncommon to know what this parameter is and thus it needs to be estimated. 
While there are some methods for estimating the additive noise parameters~\cite{Czesla2018, Urbanowicz2003, Hu2004}, we introduce a new method utilizing the relationship between the sublevel set lifetimes from both the signal and noise and the additive noise distribution~parameters.

To generate a theoretical relationship between the mean lifetime $\mu_L$ and the
distribution parameters, we recall \eqref{muL_expression_final}:
\begin{equation*}
    {\mu_L = 3\int_{-\infty}^{\infty} F(z)\left[ 1-F(z) \right] dz.}
\end{equation*}
In the subsequent subsections, we show how this relationship is used for each of
the four noise models analyzed in this work. However, when the signal is not
pure noise, which would be the case for any informative time series, the mean
lifetime is heavily influenced from the lifetimes associated to significant
features. To address this issue, we instead calculate the median of the
lifetimes as it is robust up to $50\%$ outliers and apply a signal compensation. In this sense outliers are referring to the persistence pairs associated to signal that are greater than the cutoff. This brings up an assumption for this
distribution parameter estimation method to function correctly: the number of
persistence diagram features associated with noise $N_n$ must be equal to or
greater than the number of features from the signal $N_s$. Additionally, when $N_n$ approaches $N_s$ the cutoff becomes
more conservative due to the robustness limitation of the median. To minimize this effect, in
\secref{signal_compensation}, we develop a numeric compensation
multiplier which uses the persistence pairs associated to both additive noise and signal. In general, the condition for $N_s < N_n$ is met if the time
series is sampled at a rate sufficiently higher than the Nyquist sampling
criteria $f_{\rm Nyquist}$ and, of course, the time series has some additive
noise. If these conditions are not met, we suggest the use of an
alternative method to estimate the distribution parameter of the additive noise
and apply its associated cutoff equation in
\secref{noise_models_applied}.

For a symmetric distribution of the lifetimes, the median would be an accurate
estimate of the mean. However, for most additive noise distributions (e.g. Gaussian), the distribution of the resulting sublevel set
persistence lifetimes is not symmetric. Therefore, we resort to approximating the relationship between the
mean and median numerically. While there are methods to estimate the mean using
the median and Inter-Quartile Range (IQR) as described in~\cite{Wan2014}. This
method is only robust for up to $25\%$ outliers due to the $Q_3$ upper quartile. Therefore, we use the numerically approximated
ratios of $\rho = \bar{L} / \tilde{L}$ as provided in Table~\ref{tab:r_ratios} for
each of the four distributions investigated, where $\bar{L}$ is the sample mean
lifetime and $\tilde{L}$ is the sample median lifetime. For each of these
numeric estimates a time series of length $10^5$ was used. This numeric
experiment was repeated ten times to provide a mean $\rho$ with uncertainty. This ratio can be used to estimate the mean lifetime as $\bar{L} \approx \rho\tilde{L}$.
\begin{table}[h]
\centering
\caption{Ratios $\rho = \bar{L} / \tilde{L}$ for estimating sample mean from the sample median with uncertainty as three standard deviations}
\vspace{1.5mm}
\begin{tabular}{c|cccc}
Distribution & Gussian & Uniform & Rayleigh & Exponential \\ \hline
$\rho = \bar{L}/\tilde{L}$ & $1.154 \pm 0.012$  & $1.000 \pm 0.010$ & $1.136 \pm 0.013$ & $1.265 \pm 0.016$ \\
\end{tabular}
\label{tab:r_ratios}
\end{table}

\subsection{Relating the distribution Statistic to the Median Lifetime}
\label{ssec:median_lifetime_to_statistic}
We now apply \eqref{muL_expression_final} and $\rho$ from Table~\ref{tab:r_ratios}
to find relationships between the median lifetime $M_L$ and the distribution
parameter used in each distribution's cutoff equation.

\mypara{Normal Distribution:}
For estimating $\sigma$ of the Gaussian distribution, we use
\eqref{muL_expression_final} and the Gaussian cumulative distribution to
estimate $\mu_L$ as a function of $\sigma$. Specifically, by numerically
approximating the integral in \eqref{muL_expression_final} using $z\in [-10,
10]$ with $10^6$ uniformly spaced samples, we found the relationship
\begin{equation}
\sigma \approx \frac{\mu_L}{1.6925}.
\label{eq:normal_rel}
\end{equation}
We then use $\rho$ to have Eq.~\eqref{normal_rel} as a function of the median lifetime $M_L$ as
\begin{equation}
\sigma \approx \frac{\rho M_L}{1.692} \approx 0.680M_L,
\label{eq:normal_relationship}
\end{equation}
where $M_L$ is the median lifetime. Applying this result to
\eqref{cutoff_normal} allows for a cutoff to be calculate as
\begin{equation}
    C_\alpha \approx 1.923 \tilde{L} \: {\rm erf}^{-1}\left[2(1-\sqrt{\alpha})^{1/n}
    - 1\right],
    \label{eq:cutoff_est_normal}
\end{equation}
where $\tilde{L}$ is the sample median lifetime.

\mypara{Uniform Distribution}
Next, we apply \eqref{muL_expression_final} to the uniform cumulative
distribution to estimate $\mu_L$ as a function of $\Delta$. Substituting
\eqref{unif_F_norm_changed} into  \eqref{muL_expression_final} results in
\begin{equation*}
    \mu_L = 3\int_{-\Delta/2}^{\Delta/2} \frac{2z + \Delta}{2\Delta}\left[
        1-\frac{2z + \Delta}{2\Delta} \right] dz.
\end{equation*}
Expanding and solving this integral results in the relationship
\begin{equation}
    \mu_L = \frac{\Delta}{2} \implies \Delta = 2\mu_L = 2M_L.
    \label{eq:uniform_relationship}
\end{equation}
Applying this result to \eqref{cutoff_uniform} allows for a cutoff to be calculate as
\begin{equation}
    C_\alpha = 2 \tilde{L} \left[ {2\left(1-\sqrt{\alpha}\right)}^{1/n} - 1 \right].
    \label{eq:cutoff_est_uniform}
\end{equation}

\mypara{Rayleigh Distribution}
For estimating $\sigma$ of the Rayleigh distribution, we again use
\eqref{muL_expression_final} with the cumulative Rayleigh distribution to numerically
estimate the relationship between $\mu_L$ and $\sigma$ as
\begin{equation}
\sigma \approx \frac{\mu_L}{1.102} \approx \frac{\rho M_L}{1.102} \approx 1.025 M_L,
\label{eq:rayleigh_relationship}
\end{equation}
where the integral in \eqref{muL_expression_final} was numerically approximated
using $z\in [0, 20]$ with~$10^6$ uniformly spaced samples.
Applying this result to \eqref{cutoff_rayleigh} allows for a cutoff to be calculate as
\begin{equation}
C_\alpha \approx 1.025 \tilde{L} \left( \sqrt{-2 \ln\left([1-\sqrt{\alpha}]^{1/n}\right)} - \sqrt{-2 \ln\left(1-[1-\sqrt{\alpha}]^{1/n}\right)}  \right).
\label{eq:cutoff_est_ray}
\end{equation}

\mypara{Exponential Distribution}
Next, we apply \eqref{muL_expression_final} to the exponential cumulative
distribution function to estimate $\mu_L$ as a function of $\lambda$.
Substituting \eqref{exp_F_norm} into  \eqref{muL_expression_final} results in
\begin{equation}
\mu_L = 3\int_{0}^{\infty} \left(1-e^{-\lambda z}\right)e^{-\lambda z} dz,
\end{equation}
which was solved using a $u$-substitution as
\begin{equation}
\mu_L = \frac{3}{2\lambda} \rightarrow \lambda = \frac{3}{2\mu_L} .
\label{eq:uniform_rel}
\end{equation}
By then using the appropriate $\rho$ from Table~\ref{tab:r_ratios} to use $M_L$
instead of $\mu_L$, we approximate $\lambda$ from the median lifetime:
\begin{equation}
\lambda \approx \frac{1.875}{M_L}.
\label{eq:uniform_relationship}
\end{equation}
Applying this result to \eqref{cutoff_exp} allows for a cutoff to be calculate as
\begin{equation}
C_\alpha \approx -0.533{\tilde{L}} \ln \left(  [1-\sqrt{\alpha}]^{1/n} - [1-\sqrt{\alpha}]^{2/n}  \right).
\label{eq:cutoff_est_exp}
\end{equation}

\section{Signal Compensation for the Cutoff and Distribution Parameter} \label{sec:signal_compensation}
In this section, we discuss the effects of signal on the cutoff estimation methods described.
In \ssecref{median_lifetime_to_statistic} we assumed that the time
series was of the form $\epsilon = \{ \epsilon_1, \epsilon_2, \ldots, \epsilon_n \} \stackrel{iid}{\sim} \mathcal{N}$. However, in practice we typically
have some underlying informative signal $s \colon \R \to \R$ and have a
time series of the form $x(t) = s(t) + \epsilon$ with a finite domain as $t \in [t_a, t_b]$.
The resulting sublevel sets from $s(t) + \epsilon$ are assumed to have some lifetimes from $s(t)$ with the slope of the signal having an effect on the
lifetimes associated with~$\mathcal{N}$.
Because of these effects, we attempt to compensate the cutoff calculation and
distribution parameter estimations for these effects for a general signal. Since
a general signal is, in practice, rather subjective, we move away from a
theoretical analysis of the signal and rather analyze the effects of the signal
experimentally.
We have partially addressed this issue of signal compensation by implementing the median lifetime $M_L$ instead of the mean lifetime $\mu_L$ with the median being an outlier robust statistic for up to 50\% outliers.
Even with the use of the median, we need to further develop a signal compensation procedure to improve the accuracy of the suggested cutoff.

\begin{figure}[h]
    \centering
    \includegraphics[scale = 0.29]{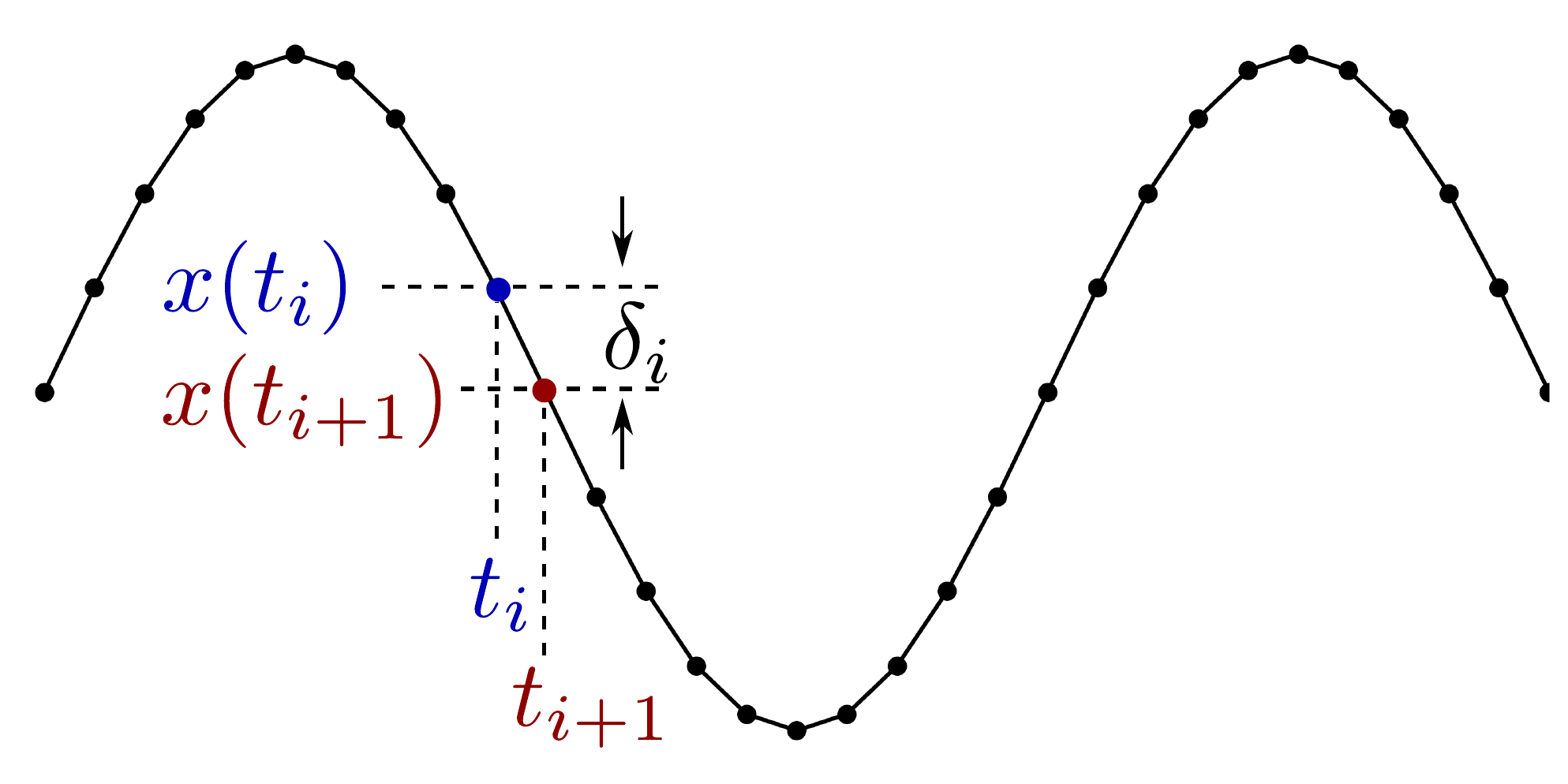}
    \caption{Example time series showing sample $\delta_i$.}
    \label{fig:example_delta_i_fig}
\end{figure}
To fully understand the effects of signal on estimating the cutoff, we do a
numeric study to develop a method for adjusting the median lifetime such that
$R M_L(s(t) +\epsilon) \approx M_L(\mathcal{N})$, where $R$ is the compensation term.
This analysis requires a new variable which we term~$\delta$ as the
median step size with $\delta_i = x(t_{i+1}) - x(t_i)$ as shown in
\figref{example_delta_i_fig}, where $x(t)$ is a discretely and uniformly
sampled signal with a constant sampling rate $f_s$.
We now experimentally approximate the effects of signal on the median lifetime
by using three ``generic" signals suggested by~\cite{Wood1996} as
\begin{equation}
    f_1(t) = t - t^3/3,
\end{equation}
with $t \in [3.1, 20.4]$ and sampling rate $f_s = 20$ Hz,
\begin{equation}
    f_2(t) = \sin(t) + \sin(2t/3),
\end{equation}
with $t \in [3.1, 20.4]$ and sampling rate $f_s = 20$ Hz,
and
\begin{equation}
    f_3(t) = -\sum_{i=1}^5 \sin((i+1)t + i),
\end{equation}
with $t \in [-10, 10]$ and sampling rate $f_s = 20$ Hz.
Additionally, additive noise is included in the signal with $s(t) = A f(t) +
\epsilon$ with the additive noise distribution parameter set to one (e.g.
$\sigma = 1$ for Gaussian) and signal amplitude~$A$
increment by unit steps starting from zero such that the $\delta$ is also incremented until reaching a value $\delta/\sigma =2$. At each $\delta$ we calculate the median
lifetime $\tilde{L}$ for 100 trials to provide a mean $\tilde{L}$ with
uncertainty $u_L$ as one standard deviation (see
\figref{fitting_to_sine_median_lifetime_Gaussian} for the Gaussian
additive noise example).

Our goal is to find a function to approximately fit this relationship
between $\delta$ and $\tilde{L}$ for each distribution type. By observation of
the median lifetimes in \figref{fitting_to_sine_median_lifetime_Gaussian},
we experimentally found an approximate template function:
\begin{equation}
\tilde{L}^* = \tilde{L}_0 e^{-c_1{\left(\frac{\delta}{\delta+\tilde{L}}\right)}^{c_2}},
\label{eq:fitted_L_star}
\end{equation}
where $\tilde{L}_0$ is the median lifetime when $\delta = 0$ or when the signal
is just additive noise $\mathcal{N}$.

\begin{figure}[h]
    \centering
    \includegraphics[width = \textwidth, center]{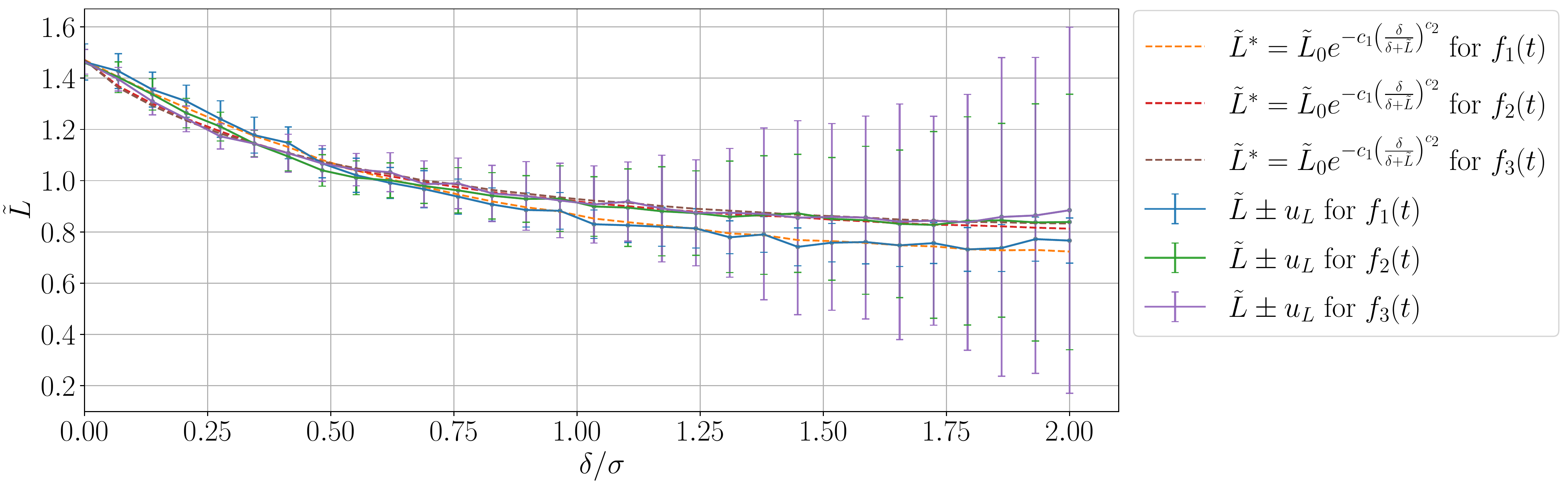}
    \centering
    \caption{Numeric function fitting of \eqref{fitted_L_star} to the mean of the median lifetime $\tilde{L}$ of $f_i(t)$ for $i \in [1,3]$ where $\mathcal{N}$ is unit variance Gaussian additive noise with $\delta \in [0,2]$ being incremented to understand the effects of signal on the median lifetime.}
    \label{fig:fitting_to_sine_median_lifetime_Gaussian}
\end{figure}

As shown in \figref{fitting_to_sine_median_lifetime_Gaussian}, the fitted function shows a very similar quality to the numerically simulated means of the median lifetimes when the two constants in \eqref{fitted_L_star} were set to $c_1 \approx 0.845$ and $c_2 \approx 0.809$ for a Gaussian additive noise, which were chosen using the BFGS minimization of the $\ell_2$ norm cost function on the residuals when fitting to $\tilde{L}$ for all three generic functions.
Another characteristic of these constants is that they are approximately independent of the additive noise distribution parameter, sampling frequency, and time series.
The two constants from \eqref{fitted_L_star} are provided in Table~\ref{tab:constants_from_fitting} for the four distributions investigated in this work.
\begin{table}[h]
\centering
\caption{Constants of \eqref{fitted_L_star} for each distribution type
    investigated in this work with associated uncertainty from ten trials.}
\vspace{1.5mm}
\begin{tabular}{c|cccc}
Distribution & Gussian & Uniform & Rayleigh & Exponential \\ \hline
$c_1$ & $0.845 \pm 0.029$ & $0.880 \pm 0.017$ & $0.726 \pm 0.026$ & $0.436 \pm 0.036$ \\
$c_2$ & $0.809 \pm 0.061$ & $0.639 \pm 0.026$ & $0.605 \pm 0.054$ & $0.393 \pm 0.075$ \\
\end{tabular}
\label{tab:constants_from_fitting}
\end{table}
With these constants, we calculate a multiplication compensation term for the signal as $R$, which is calculated from \eqref{fitted_L_star} as
\begin{equation}
R =\frac{\tilde{L}_0}{\tilde{L}^*} = e^{c_1{\left(\frac{\delta}{\delta+\tilde{L},}\right)}^{c_2}}
\label{eq:R_equation}
\end{equation}
which is used to compensate for the effects of signal with $C_\alpha^* = R C_\alpha$ and $\sigma^* = R \sigma$.

Unfortunately, when $s(t)$ is unknown, the $\delta$ parameter used in
\eqref{R_equation} can no longer be directly calculated from the time series or
sublevel set persistence diagram. To approximate $\delta$ we use the lifetimes
greater than the initial uncompensated cutoff $C_\alpha$ as
\begin{equation}
\delta \approx \frac{2}{n}\sum L_{C_\alpha},
\label{eq:Delta_approx}
\end{equation}
where $L_{C_\alpha}$ are the lifetimes greater than $C_\alpha$.

To validate the accuracy of \eqref{R_equation} with $\delta$ approximated from \eqref{Delta_approx} we estimate $\sigma$ with and without the signal compensation $R$ from \eqref{R_equation}, we use a new time series $x(t) = A\sin(\pi t) +\epsilon$ with $\mathcal{N}$ being a Gaussian distribution with unit variance and $A$ incremented to change $\delta \in [0,2]$ for $100$ trials at each~$\delta$. As shown in
\figref{scale_amp_with_param_approx_comp_to_old_and_sine_and_slope_Gaussian},
the true $\sigma = 1$ and the estimated $\sigma$ without compensation from
\eqref{normal_relationship} shows an underestimate as $\delta$ increases until
plateauing around $\delta/\sigma \approx 1$, which would cause for a cutoff that
may not capture all of the lifetimes associated with noise. However, the signal
compensated distribution parameter $\sigma^*$ shows an accurate estimation of
$\sigma$ even as $\delta$ becomes significantly large. This example demonstrates
the importance of signal compensation for an accurate cutoff and distribution
parameter estimation.

\begin{figure}[h]
    \centering
    \includegraphics[width = \textwidth, center]{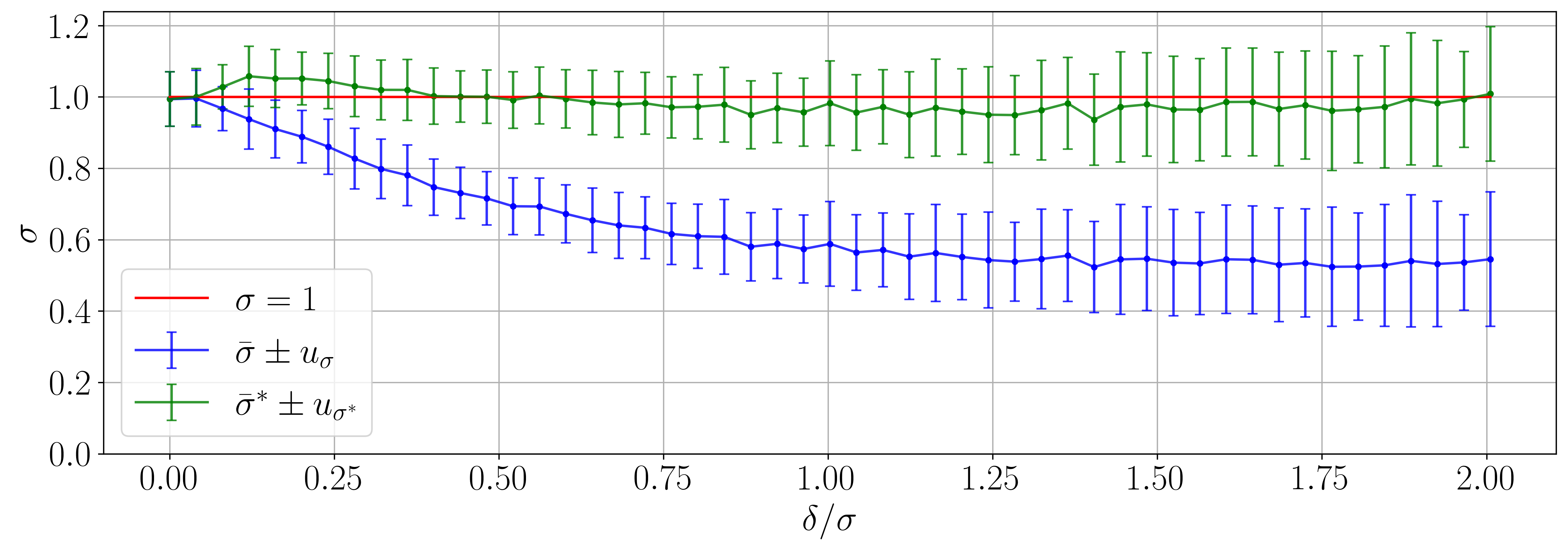}
    \centering
    \caption{Demonstration of distribution parameter $\sigma$ estimation of Gaussian additive noise in $x(t) = A\sin(\pi t) + \mathcal{N}$ using the median lifetime with and without signal compensation as $\sigma$ and $\sigma^*$, respectively.}
    \label{fig:scale_amp_with_param_approx_comp_to_old_and_sine_and_slope_Gaussian}
\end{figure}


\section{Comparisons and Examples} \label{sec:examples}
We first make a comparison between our ANAPT method for estimating a suitable cutoff in the sublevel set persistence diagram to two common approaches. Namely, persistent entropy for labeling persistence pairs as signal or noise based on their relative entropy and bootstrap resampling to generate a cutoff based on the bottleneck distance between persistence diagrams.
In our second analysis we experimentally validate our ANAPT method for three examples: (1) a numeric
simulation of the noise types with their associated cutoffs calculated;~(2) a
numerically generated quasiperiodic time series with additive noise; and~(3)~experimental data from a pendulum with natural noise of an unknown distribution.
These three examples demonstrate the wide breadth of noise distributions for
selecting a cutoff, the effects of signal in estimating an appropriate cutoff,
and the ability for the cutoff estimation method to function for experimental
data with an unknown distribution~type.

\subsection{Comparison to Standard Methods} \label{ssec:comparison}
In this section, we analyze the effectiveness of our ANAPT method
of the sublevel sets for estimating a cutoff compared to other standard
techniques for separating topological noise from signal.
For our first comparison we separate topological noise from signal using
persistent entropy~\cite{Rucco2016,Atienza2017}, bootstrap resampling of the bottleneck
distance~\cite{Fasy2014}, and ANAPT methods for an
example~signal. 

\mypara{Persistent Entropy}
Persistent entropy uses the Shannon entropy of the lifetimes from a persistence
diagram to measure the significance of a persistence pair~\cite{Rucco2016}. By comparing the
entropy of a modified set of lifetimes to a relative entropy level, the
identification of a feature as either noise or signal is determined.
Details on how to implement persistent entropy for separating signal from
topological noise can be found in ~\cite{Atienza2017}.

\mypara{Bootstrap Resampling}
Bootstrap resampling is based on resampling a dataset with replacement; here, we
bootstrap residuals to estimate the additive noise. Since we are working with dependent
time series data, we use a frequency domain function fitting method to
calculate these residuals.
The bootstrap resampling procedure is as follows:
\begin{enumerate}
    \item The data is a sequence of real-valued observations, which we can model
        as a function~$x \colon \{t_1,t_2, \ldots, t_n\} \to \R$. For
        each $t\in \{ t_1, t_2,\ldots t_n\}$, this reading is actually a sum of a true
        underlying signal and some (zero-centered) additive noise: $x(t) = s(t) + \epsilon$.
        However, we do not know
        exactly what $s(t)$ and $\epsilon$ are individually.
    \item Compute $\dgm{x}$.
    \item Filter the signal using a Butterworth low-pass filter
        of the Fast Fourier Transform (FFT) with a frequency
        threshold of twice the sampling rate. The resulting filtered signal $\hat{s}(t)$ is an approximation of $s(t)$.
    \item For each $t$, set $\hat{\epsilon}(t) = \hat{s}(t) - x(t)$ as an estimate of $\epsilon$
    \item Sampling over all $\hat{\epsilon}$ with replacement, we can obtain:
        $\epsilon^*_1, \epsilon^*_2, \ldots,  \epsilon^*_n$  (our
        bootstrapped noise).\label{step:get-noise}
    \item Our new signal is $x^*(t) = \hat{s}(t) +
        \epsilon^*$.\label{step:new-signal}
    \item Compute the persistence diagram $\dgm{x^*}$.
    \item Compute the bottleneck distance between $\dgm{x}$ and
        $\dgm{x^*}$.\label{step:bootstrap-dist}
    \item Repeat Steps~\ref{step:get-noise} through~\ref{step:bootstrap-dist} to
        get $N$ bootstrapped distances. (In our experiments, we used $N=1000$.
\end{enumerate}

At the end of this procedure, we have a set of $N$ bootstrapped distances
between the persistence diagram computed from the data, $\dgm{x}$, and the
diagrams obtained through the bootstrap; let~$\{d_1^*, d_2^*, \ldots, d_{N}^*
\}$ be this set of distances.  Letting $\alpha=0.05$, let $\bar{d}^*_{95\%}$
denote the $95\%$-ile of these distances.  Then, with $95\%$ confidence, we can
say that $d(\dgm{x},\dgm{s}) \leq \bar{d}^*_{0.05}$.  Hence, for each point in
$\dgm{x}$, we can classify it as either \emph{true signal} or
\emph{indistinguishable from noise} by thresholding the liftetime of the points
by $2 \bar{d}^*_{0.05}$.\footnote{The factor of $2$ appears, since the lifetime
of a persistence point is half of its $L_{\infty}$-distance to the diagonal.} It should be noted that for each resampling the end conditions are held constant to make the bottleneck distance correspond to additive noise only.

\mypara{Data}
The example signal used to analyze the effectiveness of each method is a chaotic
solution to the Lorenz system
\begin{equation}
    \frac{dx}{dt}   = \sigma (y-x), \: \frac{dy}{dt}   = x (\rho -z) - y, \: \frac{dz}{dt}   = xy - \beta z.
    \label{eq:lorenz}
\end{equation}
The Lorenz system was sampled at a rate of 200 Hz for 200 seconds with system
parameters~$\sigma = 10.0$, $\beta = 8.0 / 3.0$, and $\rho = 181.0$. Only the
last 2500 samples or 12.5 seconds of the signal were used to avoid transients.
Additive Gaussian noise with a Signal to Noise Ratio (SNR) of 23 dB was used to
contaminate the signal. For our analysis, we used the time series $x(t)$ as shown
in the top of \figref{modes_of_failure_for_other_methods}.

\begin{figure}[h]
    \centering
    \includegraphics[width = \textwidth, center]{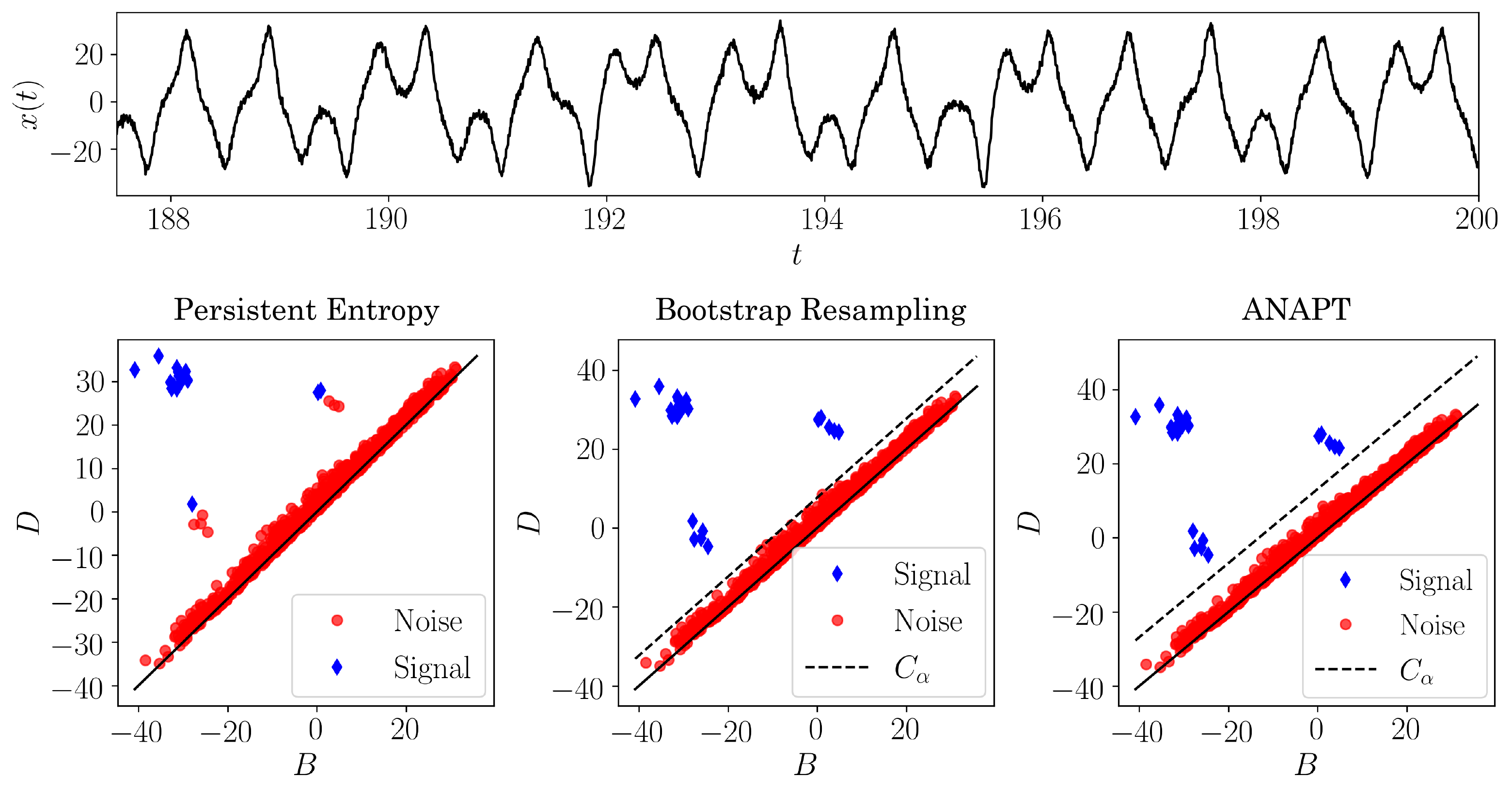}
    \caption{Comparison between the persistent homology, bootstrap resampling,
    and ANAPT methods for separating noise from
    topological signal in the sublevel set persistence diagram. The example
    signal is generated from a chaotic Lorenz system.}
    \label{fig:modes_of_failure_for_other_methods}
\end{figure}

\mypara{Comparison}
\figref{modes_of_failure_for_other_methods} shows that the persistent entropy
method (bottom left subfigure) does not accurately separate all of the
persistence pairs associated to noise from those of signal. This is due to its
design being for persistence diagrams with relatively few significant points in
the persistence diagram. However, we empirically observed that the persistent entropy method has a limitation on the ratio of persistence pairs associated to signal to those associated to noise as demonstrated in Fig.~\ref{fig:modes_of_failure_for_other_methods}. This reduces its functionality for persistence diagrams where the number of features associated to signal is close to the number associated to noise.

The second method, bootstrap resampling (bottom center
subfigure in \figref{modes_of_failure_for_other_methods}), has a resulting cutoff ($C \approx 7.70$) that narrowly encapsulates all of the features associated to noise. 
We hypothesize that this is due to the bottleneck distance minimum pairing between persistence diagrams is not necessarily matching persistence pairs that are associated but rather the closest persistence pair. 
Unfortunately, the bootstrap method is also significantly more computationally demanding than the other techniques. This is due to the need to calculate the bottleneck distance for each persistence diagram pair and the need to calculate the sublevel set persistence for the number of resamplings desired (1000 for our example).  

Based on the results in \figref{modes_of_failure_for_other_methods}, it is clear that the ANAPT method is optimal for estimating a cutoff for seperating noise from
topological signal. Our ANAPT method accurately seperates all persistence pairs associated to noise from signal with a cutoff $C_\alpha \approx 12.96$. This is in comparison to the optimal cutoff of $C_\alpha \approx 12.70$ based on knowing $\sigma$ and the theory in Section~\ref{ssec:extrema_distributions_for_cutoff}.

\begin{figure}[h]
    \centering
    \includegraphics[width = 0.9\textwidth, center]{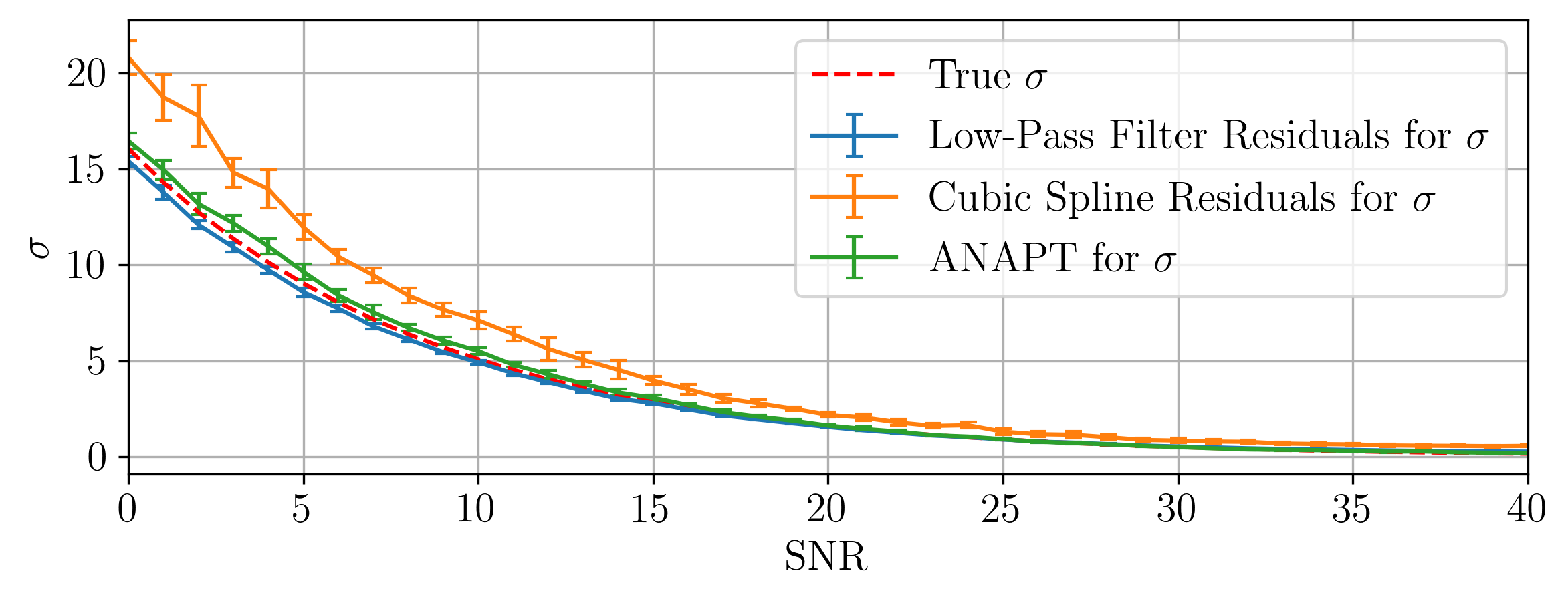}
    \caption{Comparison of distribution parameter estimation techniques (low-Pass filter residuals, cubic spline residuals, and sublevel set persistence) for estimating $\sigma$ of the additive Gaussian noise contaminating the $x(t)$ solution of the Lorenz system in Eq.~\ref{eq:lorenz}.}
    \label{fig:sigma_estimation_methods}
\end{figure}

\mypara{Estimating Standard Deviation of (Gaussian) Additive Noise}

To make further conclusions on the comparative performance of our ANAPT method, we evaluate residual based methods for estimating the standard deviation of the additive Gaussian noise. 
The two methods we implement are a low-pass Butterworth Filter with a frequency cutoff of twice the sampling frequency and a cubic spline fit on a down sampled version of the signal for estimating the residuals from the noisy signal.
The downsampling rate for the cubic spline fitting was determined such that the time delay estimated from the autocorrelation function~\cite{Box2015} is two, which avoids aliasing and over fitting.
To make comparisons between the residual based methods and the ANAPT method we again use the chaotic Lorenz signal from \eqref{lorenz}. 
However, we vary the SNR from $0$ to $40$~dB for
Gaussian additive noise. An SNR of~$0$~dB represents a noise amplitude equivalent
to the amplitude of the signal, while a value of~$40$~dB represents a relatively
low amount of additive noise. Our results are shown in
\figref{sigma_estimation_methods}. For each SNR 10 trials are run with the
mean and standard deviation as the error~bars.
As demonstrated in \figref{sigma_estimation_methods}, the cubic spline
method for function fitting and analyzing the residuals tends to overestimate
$\sigma$. In comparison, both the low-pass Butterworth filter and the ANAPT methods closely estimate $\sigma$. However the low-pass
Butterworth filter tends to slightly underestimate $\sigma$. This
underestimation can cause for a less conservative cutoff in comparison to the
slightly overestimated $\sigma$ from the ANAPT method. In our
opinion, either the low-pass filter or the ANAPT methods
could be used, but due to the additional computation of filtering the signal, we
suggest using our ANAPT method. It could be beneficial to
calculate both~$\sigma$ estimates to verify that the estimate is accurate.

\subsection{Numerically Simulated Noise Models}

The first example is of time series of only additive noise such that $\epsilon = \{ \epsilon_1, \epsilon_2, \ldots, \epsilon_n \} \stackrel{iid}{\sim} \mathcal{N}$. For this example, we simply
calculate the cutoff from the distributions corresponding cutoff Equation
(\eqref{cutoff_normal},~\eqref{cutoff_uniform},~\eqref{cutoff_rayleigh},
and~\eqref{cutoff_exp}) and show that the cutoff is greater than all sublevel
set lifetimes (see \figref{histos_with_cutoff}).
For this example, the distribution parameters are $\sigma = \Delta = \lambda = \sigma = 1$ for
the Gaussian, uniform, Rayleigh, and exponential distributions, respectively.
\figref{histos_with_cutoff} shows the histogram estimates of the probability
densities of the additive noise types on the top row, and on the bottom row are
the approximate probability distributions of the corresponding lifetimes. For
each of the lifetime distributions, the calculated cutoff~$C_\alpha$ is shown, where $\alpha = 0.001$ and $n = 10^5$ data points.

\begin{figure}[h]
    \centering
    \includegraphics[width = \textwidth, center]{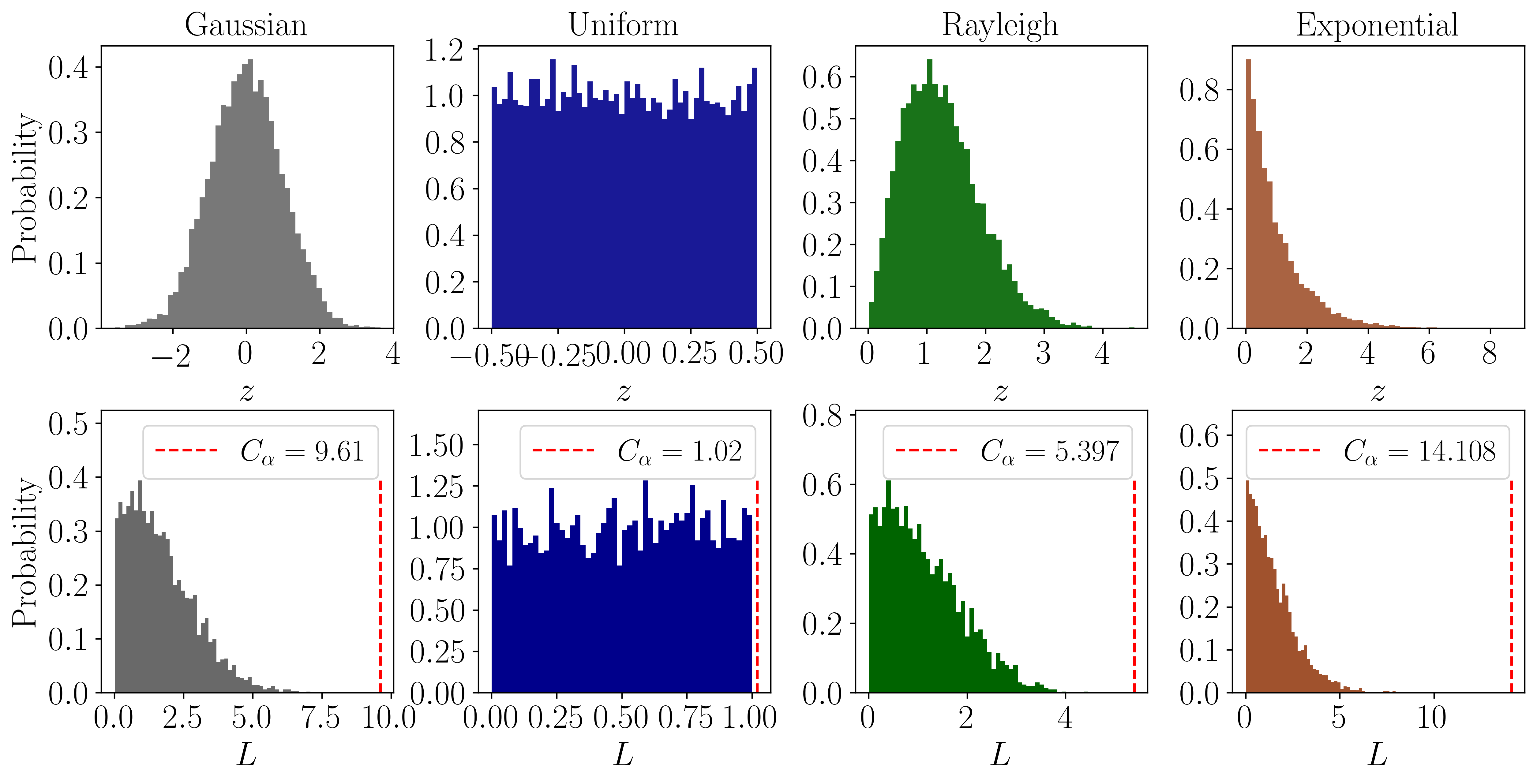}
    \caption{Numerical simulations with $10^5$ data points for each of the four
    investigated noise models with probability densities on the top row and
    lifetime probability densities on the bottom row with their associated
    cutoff $C_\alpha$. The distribution parameters were set as $\sigma = \Delta
    = \lambda = \sigma = 1$ and were estimated as $\sigma \approx 0.996$, $\Delta \approx
    1.013$, $\lambda \approx 1.009$, and $\sigma \approx 0.998$ for the Gaussian, uniform,
    Rayleigh, and exponential distributions, respectively}
    \label{fig:histos_with_cutoff}
\end{figure}

As shown, the cutoff is greater than the maximum lifetime in the probability
density histograms of the lifetimes for each type of additive noise
investigated, which demonstrates that this method for selecting a cutoff using
the statistics of the sublevel set lifetimes is suitable for selecting a cutoff
if the distribution is known. However, this example does not show what effects
dependency in the time series has on the distribution. In the next example, we examine the effect of a dependency or signal in the time series on the cutoff calculation and its accuracy.

\subsection{Numerically Simulated Signal with Additive Noise}
The second example used is a sinusoidal time series with additive
noise (see left column of \figref{signal_with_noise}) calculated~as
\begin{equation}
x(t) = A(\sin(\pi t) + \sin(t)) + \epsilon,
\label{eq:example_signal}
\end{equation}
where $A = 10$ is the signal amplitude, $\epsilon$ is additive noise with a
normal distribution with zero mean $\mu = 0$ and standard deviation $\sigma \in [0.2, 1, 2, 4]$, and $t \in [0,15]$ with a sampling rate of $f_s = 40$~Hz.
\begin{figure}[h]
    \centering
    \includegraphics[width = 0.7\textwidth, center]{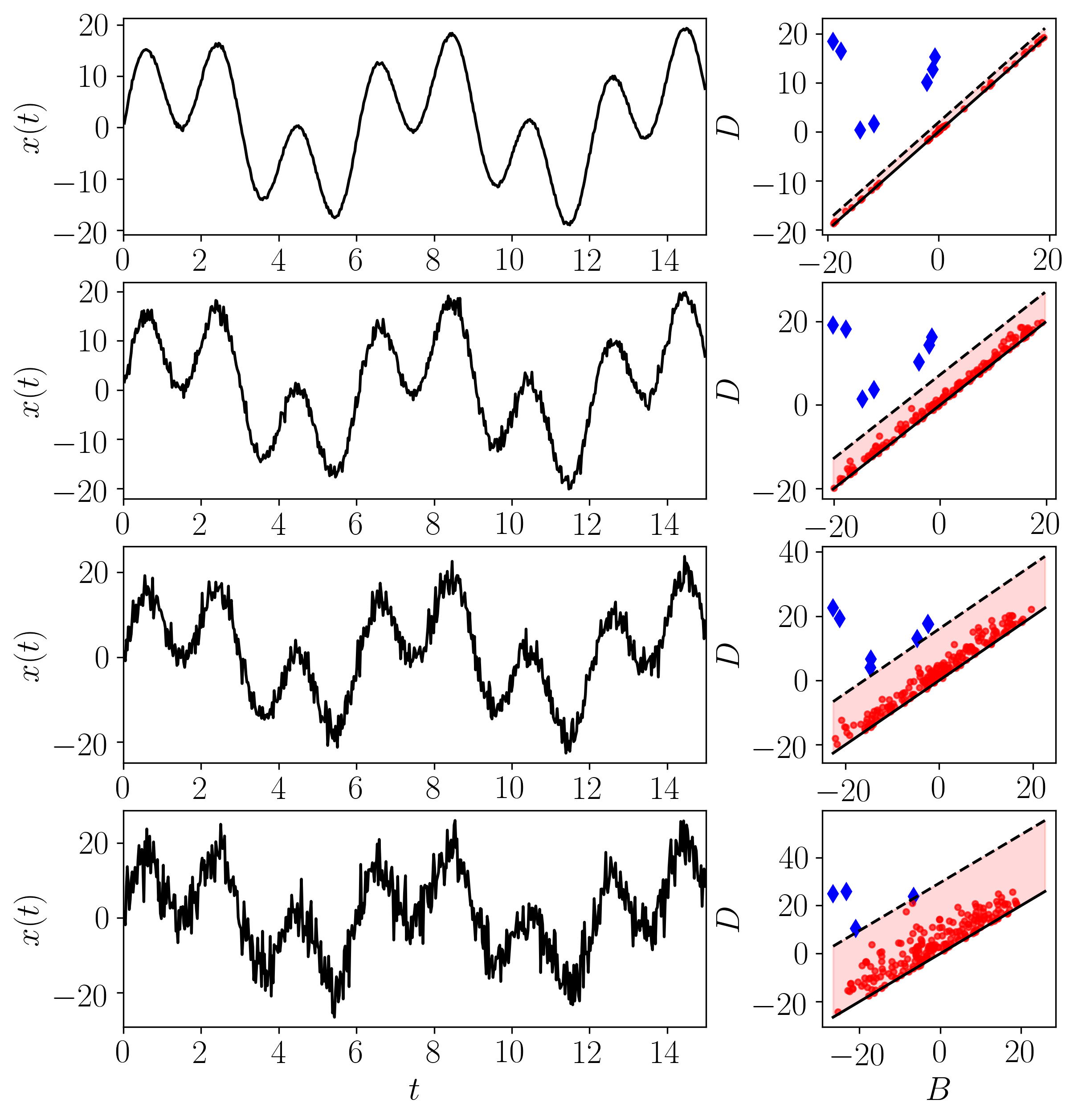}
    \caption{Example cutoff calculation for time series $x(t) = A(\sin(\pi t) +
    \sin(t)) + \epsilon$, where $\epsilon$ is additive noise from a
    Gaussian distribution with zero mean and four different standard deviations as $\sigma \in [0.2, 1, 2, 4]$. The
    resulting sublevel set persistence diagrams with cutoff $C_\alpha^*$ are shown to the right.}
    \label{fig:signal_with_noise}
\end{figure}

For a signal like this there is a need to multiply by a compensation term to the
cutoff and distribution parameter as discussed in
\secref{signal_compensation}.
This is due to the slope of the function reducing the lifetimes associated to
additive noise and the effect of lifetimes associated to signal.
For the signal with $\sigma =1$ and calculating the median lifetime and setting $\alpha= 0.001$, we calculate a
cutoff $C_\alpha \approx 5.36$ and estimated $\sigma \approx 0.751$ for this
example.
However, if we apply the compensation multiplier $R$ from \eqref{R_equation} and
approximate delta from \eqref{Delta_approx} as $\Delta \approx 0.953$ we find a
cutoff $C_\alpha^* \approx 7.54$ and $\sigma^* \approx 1.05$, which is a much
more accurate cutoff and approximate distribution parameter with the actual
$\sigma = 1$.
Similar improvements were found for the other noise levels when incorporating signal compensation. 

As shown in the example with $\sigma = 4$ at the bottom of Fig.~\ref{fig:signal_with_noise}, the cutoff will eventually encapsulate persistence pairs associated to signal if the amplitude of the noise is significant in comparison to the amplitude of the signal.

This example showed how the method for selecting a cutoff is applicable to a
time series with additive noise and an underlying signal that is not simply
sinuisodal, but quasiperiodic.
However, we still knew what the additive noise distribution type was. In our
next example, we look at experimental data, where the noise does not have a known distribution.

\subsection{Experimental Example} \label{ssec:experiment_example}
The experimental data used in this example is from a single pendulum with an
encoder to track the position of the pendulum arm. A full design document for
the pendulum is available through GitHub\footnote{\url{https://github.com/Khasawneh-Lab/simple_pendulum}}.
The data collected is from a free drop response of the pendulum where the amplitude of
the pendulum swing is decreasing due to natural damping mechanisms. The recorded
data from the encoder was translated from a voltage output to radians, but this
calibration and dimensional analysis are not covered in this manuscript, but
rather in the documentation report with a full uncertainty analysis in~\cite{Petrushenko2017}. In
\figref{experiment_signal_with_noise}, a section of the time series sampled at
100 Hz is shown with a subsection highlighted in green when the time series no
longer has significant dependencies such that~$x(t) \approx \mathcal{N}$.
Additionally, in \figref{experiment_signal_with_noise}, the resulting
persistence diagram with the corresponding noise cutoff with and without signal
adjustment are shown.
\begin{figure}[h]
    \centering
    \includegraphics[width = 0.85\textwidth, center]{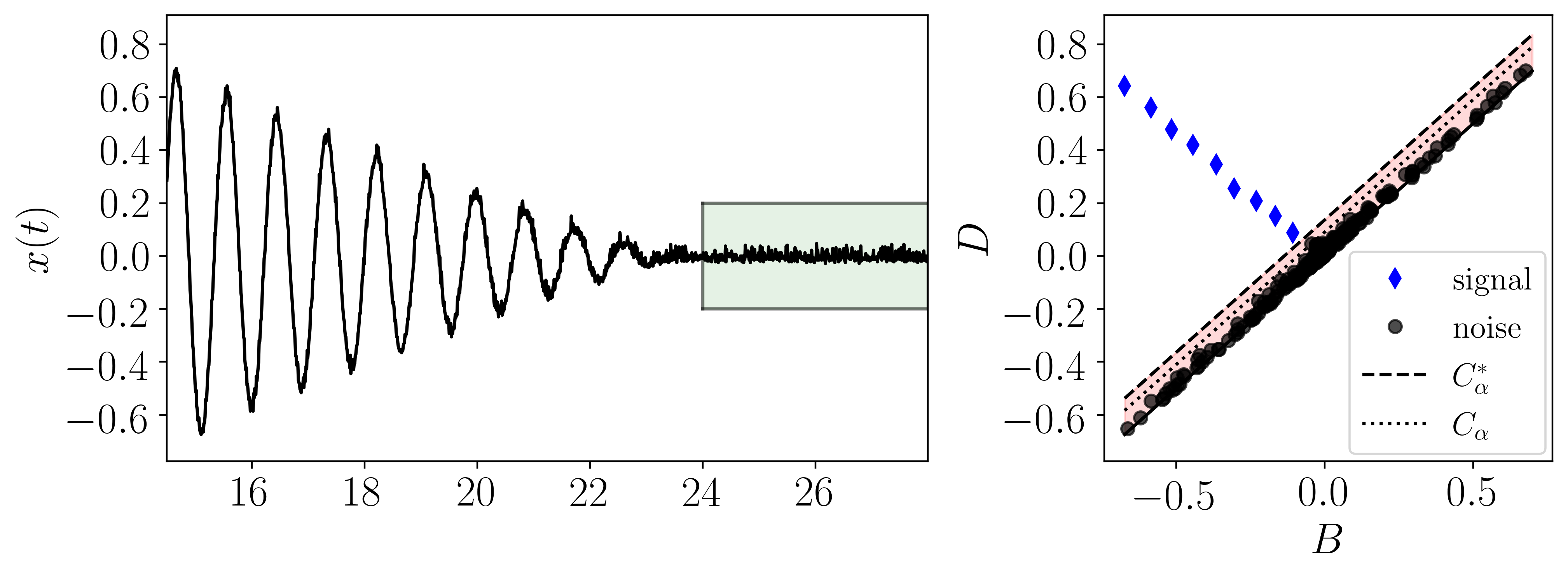}
    \caption{Experimental time series from free drop response of single pendulum
    with corresponding cutoffs in the persistence diagram and time-order
    lifetimes plot.}
    \label{fig:experiment_signal_with_noise}
\end{figure}
To generate this cutoff, an additive noise distribution model is needed. Since the
data was from an experiment this distribution is not exactly known. To better
understand and assume a distribution we used the section of the time series ($t
\in 24, 28]$) highlighted in green in the top left of
\figref{experiment_signal_with_noise} where there is approximately no dependency
in the data such that $x(t) \approx \epsilon$. Using this section of data, we
generated a histogram of the additive noise as shown in \figref{experiment_noise_distribution}.

This distribution, qualitatively, seems to best match a Rayleigh distribution,
which was used to determine the cutoff. The fitted Rayleigh distribution is also
shown in \figref{experiment_noise_distribution}, where $\sigma$ was estimated
from the median lifetime using \eqref{rayleigh_relationship} to estimate $\sigma
\approx 0.028$.
Using this fitted distribution parameter, a suitable cutoff would be $C_\alpha
\approx 0.1293$ using \eqref{cutoff_rayleigh}.
However, we are interested in calculating the cutoff using the entire signal
with $t \in [15, 28]$ as shown in \figref{experiment_signal_with_noise} since it
is not always possible to have a section of time series that does not have
signal dependencies.
If we directly apply \eqref{cutoff_rayleigh} we calculate a cutoff as $C_\alpha
\approx 0.0905$, which is lower than the desired cutoff $C_\alpha \approx
0.1293$. To account for this, we implement the signal compensation discussed in
\secref{signal_compensation} and calculated the cutoff as $C_\alpha^*
\approx 0.1365$ which is slightly conservative, but a better estimate than
without signal compensation.

\begin{figure}[h]
    \centering
    \includegraphics[scale = 0.43, center]{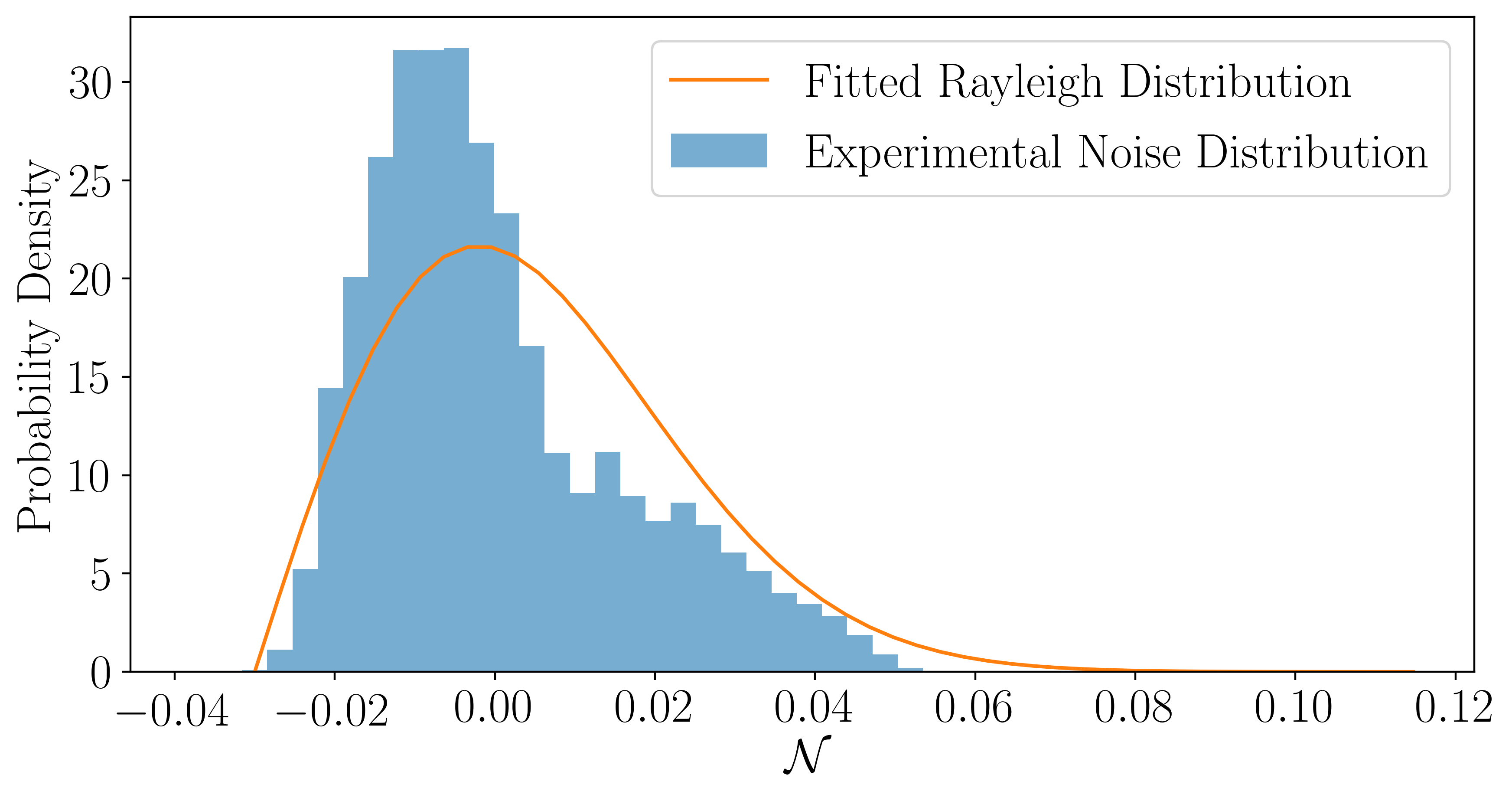}
    \caption{Histogram of the highlighted section of signal in
    \figref{experiment_signal_with_noise} with fitted Rayleigh distribution.}
    \label{fig:experiment_noise_distribution}
\end{figure}

This example highlights the importance of signal compensation in the cutoff
calculation and demonstrates the functionality of the ANAPT cutoff estimation method
for experimental data. Additionally, we showed how one of the analyzed
distributions (Rayleigh) can be used to approximate the distribution of the
additive noise and be used to calculate an accurate cutoff.

\section{Conclusion}
\label{sec:conclusion}

In this work we provided a novel statistical analysis of sublevel set
persistence diagrams from a single variable function. Specifically, we show how
to calculate a cutoff $C_\alpha$ for separating persistence pairs associated to
noise from those of signal in the persistence diagram. This method is based on
an assumption that the form of the underlying probability distribution of the
additive noise is known (e.g. Gaussian distribution). This method has several
advantages to other methods including a reduced computational time in comparison
to resampling techniques since only a single persistence diagram needs to be
calculated, no signal filtration required, widely applicable, and functions with
significant amounts of additive noise. 
Additionally, this method is essentially parameter free with the only parameter being the confidence level $\alpha$, which is suggested to be approximately 0.05. 
These features make ANAPT especially useful for applications such as extrema detection and peak prominence analysis.
While ANAPT does use TDA, it is easy to implement and use with our open-source python code available\footnote{\url{https://github.com/lizliz/teaspoon}}.

To determine the effectiveness of the cutoff, we first showed how ANAPT functions correctly for time series of pure additive noise
(i.e. $x(t) = \mathcal{N}$) with all sublevel set persistence pairs captured by
the cutoff. Next, we tested ANAPT on both an informative, simulated time
series and an experimental time series from a pendulum. For the simulated time
series we added Gaussian additive noise and showed how ANAPT accurately
determined a suitable cutoff by separating the correct sublevel set lifetimes.
For the experimental time series study, we did not know the distribution of the
additive noise and proposed approximating it as the Rayleigh distribution. This
approximation allowed for an accurate calculation of a cutoff and showed that
ANAPT functions even without the exact probability distribution of the
additive noise known.

In the future we plan to investigate a few new directions. The first is to
extend the statistical analysis to sublevel set persistence for multi-variable
functions, where a new theoretical distribution of the birth and death times
needs to be developed. The second improvement would be to develop a new method
to approximate the probability distribution function of the underlying additive
noise to alleviate the need for the distribution to be specified. The challenge
with this new development would be creating a library of distribution functions
to test against, which was out of the scope of this work.

\begin{acknowledgements}
\section{Acknowledgements}
This material is based upon work supported by the National Science Foundation
under grant numbers CMMI-1759823, DMS-1759824, and DMS-1854336.

\end{acknowledgements}

\section*{Conflict of interest}
On behalf of all authors, the corresponding author states that there is no conflict of interest.


\bibliographystyle{plain}
\bibliography{sublevel_statistics_bib}
\appendix
\section{Proofs and Algorithms}

In the appendix, we provide an omitted proof and algorithm.
Specifically, we have included the theorem showing the relationship
between the mean lifetime and mean birth and death times of a persistence
diagram and the algorithm for calculating the persistence diagram for the
sublevel-set filtration of a time series.

\subsection{Proof of Expected Lifetime Equation} \label{app:proofs}

The following proof supports a claim made in \ssecref{stats_relationship_background}.
In what follows, we use the notation $\mu_{S}$ to denote the expected value
of the distribution over the multi-set $S$.

\begin{theorem}[Expected Lifetime]\label{thm:mean-lifetime}
    Let $\mathcal{D}=\{(b_i,d_i)\}_{i=1}^n$ be a persistence diagram.
    Let~$B$,~$D$, and $L$ be the multi-sets of birth times, death times, and
    lifetimes, respectively.
    Then, the average lifetime~is:
    \begin{equation*}
        \mu_{L} = \mu_{D} - \mu_{B}.
    \end{equation*}
\end{theorem}

\begin{proof}
    By definition, $B={\{b_i\}}_{i=1}^n$, $D={\{d_i\}}_{i=1}^n$, and $L = \{ d_i-b_i\}_{i=1}^n$.
    By definition of mean and of $L$, the mean lifetime is
    \begin{equation}
    \begin{split}
    \mu_{L}
    		& = \frac{1}{n}\sum_{i = 1}^{n} (d_i - b_i).
    \end{split}
    \label{eq:L_bar_proof_step_1}
    \end{equation}
    Expanding the sum to two separate sums and using the commutative property of
    addition, we get:
    \begin{equation}
    \begin{split}
    \mu_L & = \frac{1}{n}\sum_{i = 1}^{n} (d_i - b_i)\\
    		& = \frac{1}{n}\sum_{i = 1}^{n} d_i - \frac{1}{n}\sum_{i = 1}^{n} b_i \\
    		& = \mu_{D} - \mu_{B},
    \end{split}
    \label{eq:L_bar_proof_step_2}
    \end{equation}
    where the last equality is by definition of $\mu_{D}$ and $\mu_{B}$.
    Thus, we conclude~$\mu_{L} = \mu_{D} - \mu_{B}$.
\end{proof}

\subsection{Algorithm} \label{app:algorithms}

In \algref{0dpers}, we give an $\Theta(n \log n)$ algorithm for
computing the persistence
diagram for the sublevel sets of a real-valued function.
The time series is represented as an
array $A$ of $n$ real numbers. This array is interpreted as a
continuous function~$f \colon [1,n] \to \R$, where $f(i) = A[i]$ and $f(x)$ is
calculated by linearly interpolating between~$f(\floor{x})$ and $f(\ceil{x})$.
The algorithm is similar to that presented
in~\cite{delfinado1995incremental}, but uses a priorty queue instead of the
union-find data structure.

We begin (\alglnref{defM}) by creating a linked list of local extrema in $A$,
ordered by their index in $A$.
This step takes $\Theta(n)$ time.
Each $m \in M$ stores its value or height
as~$m.height$ as well as a pointer $m.next$ to the next element of $M$.
In order to avoid a boundary effect, if the first value is a minimum, then we included it as a
minimum (resp., maximum) with height~$-\infty$ (resp., $\infty$).
We handle the last value similarly.
If $p$ is a pointer to an element of~$M$, we use~$M[p]$ to denote the
corresponding element of the linked list.

We then create a priority queue $Q$
in order to keep track of the potential min/max pairs that we could
make (where smaller value indicates higher
priority).
The priority queue stores a pair $(ptr,v)$, where $ptr$ is a pointer to an
element of $M$ and $v$ is the priority of~$M[ptr]$, in
particular,~$v=\left|M[ptr].height-M[ptr].next.height \right|$. $Q.pop$
returns the (value of the) element of $Q$ with lowest priority, and
simultaneously removes it from $Q$.
The construction of the priority queue in  \alglnref{defQ} takes $\Theta(n \log
n)$.

The while loop starting on \alglnref{while} maintains the invariant that when
entering the while loop for the $i$-th time the $i-1$ smallest persistence
points have been calculated.
The loop is straightforward, other than the
update of $Q$ and $M$ in \alglnref{update}.  To do so, we first set~$m'=m.prev$ and remove
$m$ and~$m.next$ from $M$ as well as
$m'$, $m$, and $m.next$ from $Q$.
Finally, we add $m'$ to $Q$ with priority~$|m'.height-m'.next.height|$\;
The while loop executes~$\Theta(n)$ times; therefore, the total runtime of this
algorithm is $\Theta(n \log n)$.

For our specific python implementation of this algorithm, we use the \textit{SortedList} data structure from the \texttt{Sorted Containers} Python package. Our implementation is available through the python package \texttt{teaspoon}\footnote{\label{footnote_label}\url{https://lizliz.github.io/teaspoon/}}.

\begin{algorithm}[h]\caption{Zero-Dimensional Persistence Algorithm.}\label{alg:0dpers}
    \SetAlgoLined
    \KwData{Array $A$ of $n$ Real Numbers}
    \KwResult{Persistence Diagram}
    \SetKwBlock{Beginn}{beginn}{ende}
    \Begin{
        $M =$ list of local extrema in $A$ (non-endpoint maxima and
        minima)\;\label{algln:defM}
        $Q =$ priority queue of pointers to elements of $M$\;\label{algln:defQ}
        \While{$|M|>3$}{\label{algln:while}
            $m \gets Q.pop$\;
            $b \gets \min \{ m.height,m.next.height\}$\;
            $d \gets \max \{ m.height,m.next.height\}$\;
            Add $(b,d)$ to $pairs$\;
            Update $Q$ and $M$\;\label{algln:update}
        }
        \KwRet{pairs}
     }
\end{algorithm}

\end{document}